\documentclass{amsart}
\usepackage{amssymb,url}
\usepackage{hyperref}
\hypersetup{colorlinks=true,linktocpage=true,allcolors=blue}

\theoremstyle{plain}

\newtheorem{theorem}{Theorem}[section]
\newtheorem{lemma}[theorem]{Lemma}
\newtheorem{corollary}[theorem]{Corollary}
\newtheorem{proposition}[theorem]{Proposition}
\newtheorem*{theorem*}{Theorem}
\newtheorem*{corollary*}{Corollary}
\newtheorem*{proposition*}{Proposition}

\def\thmref#1{{\theoremstyle{plain}\newtheorem*{refthm#1}{Theorem~\ref{#1}}}}
\def\propref#1{{\theoremstyle{plain}\newtheorem*{refprop#1}{Proposition~\ref{#1}}}}

\theoremstyle{definition}

\newtheorem{definition}[theorem]{Definition}
\newtheorem*{claim}{Claim}
\newtheorem*{claim1}{Claim 1}
\newtheorem*{claim2}{Claim 2}

\newtheorem{example}[theorem]{Example}

\theoremstyle{remark}

\newtheorem*{remark*}{Remark}
\numberwithin{equation}{section}

\def\proofclaim{\noindent{\it Proof of claim.}  }

\def\concat{{}^\smallfrown}
\def\upto{\upharpoonright}
\def\len#1{\lvert #1 \rvert}
\def\cl{{\rm cl}}
\def\dist{{\rm dist}}
\def\dom{{\rm dom}}
\def\ran{{\rm ran}}
\def\supp{{\rm supp}}

\def\Aut{{\rm Aut}}
\def\id{{\rm id}}
\def\NN{{\mathbb N}}
\def\ZZ{{\mathbb Z}}
\def\QQ{{\mathbb Q}}
\def\RR{{\mathbb R}}

\def\leqg{\leq_{\rm g}}
\def\leqB{\leq_{\rm B}}
\def\leqRK{\leq_{\rm RK}}
\def\leqRB{\leq_{\rm RB}}
\def\btu{\bigtriangleup}

\def\MSC#1{{\let\thefootnote\relax 
\footnote{{\it 2010 Mathematics Subject Classification.}  #1}
\setcounter{footnote}{0}}}
\def\EMAIL#1{{\let\thefootnote\relax 
\footnote{{\it E-mail address.}  \url{#1}}
\setcounter{footnote}{0}}}

\def\REVISED{{\let\thefootnote\relax
\footnote{Last revised on \today}
\setcounter{footnote}{0}}}

\begin{document}

\title{Universal subgroups of Polish groups}
\author{Konstantinos A. Beros}

\begin{abstract}
Given a class $\mathcal C$ of subgroups of a topological group $G$, we say that a subgroup $H \in \mathcal C$ is a {\em universal $\mathcal C$ subgroup} of $G$ if every subgroup $K \in \mathcal C$ is a continuous homomorphic pre-image of $H$.  Such subgroups may be regarded as complete members of $\mathcal C$ with respect to a natural pre-order on the set of subgroups of $G$.  We show that for any locally compact Polish group $G$, the countable power $G^\omega$ has a universal $K_\sigma$ subgroup and a universal compactly generated subgroup.  We prove a weaker version of this in the non-locally compact case and provide an example showing that this result cannot readily be improved.  Additionally, we show that many standard Banach spaces (viewed as additive topological groups) have universal $K_\sigma$ and compactly generated subgroups.  As an aside, we explore the relationship between the classes of $K_\sigma$ and compactly generated subgroups and give conditions under which the two coincide.
\end{abstract}

\maketitle

\setcounter{tocdepth}{1}
\tableofcontents



\section{Introduction}\label{S1}

\subsection{Background}

The study of definable equivalence relations on Polish spaces has been one of the major threads of descriptive set theory for the past thirty years.  In many cases, important equivalence relations arise from algebraic or combinatorial properties of the underlying Polish spaces.  A common situation is that of a coset equivalence relation on a Polish group $G$.  If $H \subseteq G$ is a subgroup, one defines the equivalence relation $E_H$ by 
\[
x E_H y \iff y^{-1}x \in H.
\]
Viewed as a subset of $G \times G$, $E_H$ has the same topological complexity (Borel, analytic, etc...)~as $H$ and its equivalence classes are the left cosets of $H$.  To give a concrete example, consider the equivalence relation $E_0$ on $2^\omega$, defined by 
\[
x E_0 y \iff (\forall^\infty n)(x(n) = y(n)).
\]
Identifying $2^\omega$ with the Polish group $\mathbb Z_2^\omega$, one recognizes $E_0$ as the coset equivalence relation of the subgroup
\[
\mathrm{Fin} = \{ x \in \mathbb Z_2^\omega : (\forall^\infty n) (x(n) = 0)\}.
\]

Given equivalence relations $E,F$ on a space $X$, one often asks whether or not there exists a definable map $f : X \rightarrow X$ reducing $E$ to $F$, i.e., such that
\[
(\forall x,y) (x E y \iff f(x) F f(y) ).
\]
In this situation, ``definable'' is usually (though not always) interpreted to mean Baire- or Borel-measurable.  (In the case that a Borel reduction exists, one writes $E \leqB F$.)

Returning to the setting of groups, suppose that $H,K \subseteq G$ are subgroups of a Polish group $G$ and $\varphi : G \rightarrow G$ is a group homomorphism such that 
\[ 
(\forall x) (x \in H \iff \varphi (x) \in K).
\]
This in turn gives a reduction of $E_H$ to $E_K$ since, by the properties of group homomorphisms,
\[
(\forall x,y ) (y^{-1}x \in H \iff \varphi (y)^{-1} \varphi (x) \in K).
\]
As mentioned above, one is generally interested in reducing maps which are at least Baire-measurable.  Recall, however, that Baire-measurable homomorphisms of Polish groups are automatically continuous (see Kechris \cite[9.10]{KECHRIS.dst}).  Taken together, these observations motivate the following definition.

\begin{definition}
Let $G_1,G_2$ be Polish groups.  Suppose that $H \subseteq G_1$ and $K \subseteq G_2$ are subgroups.  We say that $H$ is \emph{group-homomorphism reducible} to $K$ if, and only if, there exists a continuous homomorphism $\varphi :G_1 \rightarrow G_2$ such that $\varphi^{-1} (K) = H$.  We write $H \leqg K$.
\end{definition}

As mentioned above,
\begin{equation}\label{E1}
H \leqg K \implies E_H \leqB E_K.
\end{equation}
In fact, many Borel reductions among coset equivalence relations derive from corresponding group-homomorphism reductions.  Each of the Borel reductions $E_0 \leqB E_1,E_2,E_3$ arises in this way.\footnote{See Kanovei \cite{KAN.borel.equiv} for definitions of these equivalence relations.}  The following is a representative example.

\begin{example} Recall from above that $E_0$ is the coset equivalence relation of the subgroup $\mathrm{Fin} \subseteq \mathbb Z_2^\omega$.  Consider the equivalence relation $E_2$, where 
\[
x E_2  y \iff \sum_{x(n) \neq y(n)} \frac{1}{n+1} < \infty.
\]
Notice that $E_2$ is the coset equivalence relation of the subgroup 
\[
H = \{ x \in \mathbb Z_2^\omega : \sum_{x(n) \neq 0} \frac{1}{n+1} < \infty\}
\]  
A map witnessing the reduction $E_0 \leqB E_2$ is 
\[
\varphi (x) = x(0) \concat  x(1)^2  \concat  x(2)^4   \concat  x(3)^8 \concat \ldots.
\]  
In other words, $\varphi$ copies the $n$th bit of $x$ to a block of $2^n$ bits of $\varphi (x)$.  Observe that $\varphi$ is actually a continuous group homomorphism of $\mathbb Z_2^\omega$ and $\mathrm{Fin} = \varphi^{-1} (H)$, i.e., $\mathrm{Fin} \leqg H$.  (This follows since each nonzero bit of $x$ increases $\sum \{\frac{1}{n+1} : \varphi (x)(n) \neq 0 \}$ by more than $\frac{1}{2}$.)
\end{example}

In general, however, the converse of \eqref{E1} is false.  Consider the following situation;  suppose that $H,K$ are normal subgroups of a group $G$ and $H \leqg K$, via $\varphi$.  The map $\varphi$ induces an injective homomorphism $\tilde\varphi : G / H \rightarrow G / K$, defined by $\tilde\varphi (\pi_H(x)) = \pi_K (\varphi (x))$, where $\pi_H$ and $\pi_K$ are the quotient maps onto $G /H$ and $G / K$, respectively.  This observation justifies the following two examples.

\begin{example}  Let 
\[
H_2 = \{ x \in \mathbb Z^\omega : (\forall n) (x(n) \mbox{ is divisible by 2}) \}
\]
and
\[
H_3 = \{ x \in \mathbb Z^\omega : (\forall n) (x(n) \mbox{ is divisible by 3}) \}.
\]
Note that $\mathbb Z^\omega / H_2 \cong \mathbb Z_2^\omega$ and $\mathbb Z^\omega / H_3 \cong \mathbb Z_3^\omega$.  Thus $H_2 \nleq_{\rm g} H_3$ and $H_3 \nleq_{\rm g} H_2$, since there are no injective homomorphisms $\mathbb Z_2^\omega \rightarrow \mathbb Z_3^\omega$, or {\it vice versa}.

On the other hand, $E_{H_2} \leqB E_{H_3}$ via the map $f : \mathbb Z^\omega \rightarrow \mathbb Z^\omega$ given by
\[ 
f(x)(n) = \begin{cases}
0 \quad \mbox{if $x(n)$ is even,}\\
1 \quad \mbox{if $x(n)$ is odd,}
\end{cases}
\]
for each $n \in \omega$.  Similarly, $E_{H_3} \leqB E_{H_2}$.
\end{example}

\begin{example}  In \cite{ROSENDAL.cofinal}, Christian Rosendal showed that the coset equivalence relation of the subgroup 
\[
\mathcal B = \{ x \in \mathbb Z^\omega : (\exists M) (\forall n) (\lvert x(n) \rvert \leq M) \}
\]
is Borel-complete for $K_\sigma$ equivalence relations.  In particular, $E_H \leqB E_{\mathcal B}$, for each $K_\sigma$ subgroup of $\mathbb Z^\omega$.  There are, however, $K_\sigma$ subgroups which are not group-homomorphism reducible to $\mathcal B$.  For example, 
\[
2\mathcal B = \{ x \in \mathcal B : (\forall n) (x(n) \mbox{ is even})\} \ \nleq_{\rm g}\  \mathcal B,
\]
since $\mathbb Z^\omega / 2\mathcal B$ has elements of order 2 and $\mathbb Z^\omega / \mathcal B$ has no elements of finite order.
\end{example}

Our work on $\leqg$ was motivated in part by the last example.  In particular we wondered if there would be an analog of Rosendal's theorem for group-homomorphism reductions.  In other words, are there $\leqg$-complete $K_\sigma$ subgroups?  

Naturally, one can ask this question for classes besides $K_\sigma$.  This suggests the following definition.

\begin{definition}
Let $G_1$, $G_2$ be Polish groups and $\mathcal C$ a class of subgroups of $G_1$.  We say that a subgroup $K$ of $G_2$ is \emph{universal for subgroups of $G_1$ in $\mathcal C$} if, and only if, for each subgroup $H \subseteq G_1$, with $H \in \mathcal C$, we have $H \leqg K$. 

In the case that $G_1 = G_2$ and $K \in \mathcal C$, we simply say that $K$ is a \emph{universal $\mathcal C$ subgroup of $G_1$}.
\end{definition}

In this context, the simplest classes to study are those for which membership of a subgroup $H$ in the class is determined by the nature of a generating set for $H$.  In the present paper, we will consider two classes of subgroups with this property, namely the classes of $K_\sigma$ and compactly generated subgroups.  In \cite{BEROS.universal.projective}, we take up the study of $\mathbf \Sigma^1_n$ and $\mathbf \Pi^1_1$ subgroups.

\subsection{Summary of results}\label{summary}

Our main results concern the existence of universal $K_\sigma$ and universal compactly generated subgroups in the countable powers and products of various Polish groups.

The following is our principal positive result.

\begin{theorem}\label{T1}\thmref{T1}
Let $(G_n)_{n \in \omega}$ be a sequence of locally compact Polish groups, each term of which occurs infinitely often (up to isomorphism).  We have the following;
\begin{enumerate}
\item $\prod_n G_n$ has a universal compactly generated subgroup.
\item $\prod_n G_n$ has a universal $K_\sigma$ subgroup.
\end{enumerate}
\end{theorem}

Although stated for products, Theorem~\ref{T1} implies that the countable power of any locally compact group has universal $K_\sigma$ and compactly generated subgroups, e.g., $\mathbb Z_2^\omega$, $\mathbb Z^\omega$, $\mathbb R^\omega$, $\mathbb Q^\omega$ (with the discrete topology on $\mathbb Q$) and $\mathbb T^\omega$ (where $\mathbb T$ is the unit circle in $\mathbb C$). 

For the case of groups which are not locally compact, we have an ``approximation'' of the last theorem.

\begin{theorem}\label{T2}\thmref{T2}
For every Polish group $G$, there is an $F_\sigma$ subgroup $H \subseteq G^\omega$ which is universal for $K_\sigma$ subgroups of $G^\omega$, i.e., every $K_\sigma$ subgroup of $G^\omega$ is a continuous homomorphic pre-image of $H$.
\end{theorem}

We prove Theorems~\ref{T1} and \ref{T2} in Sections~\ref{S4.2} and \ref{S5}, respectively.  In Section~\ref{S6}, we prove the following counterpoint to Theorem~\ref{T2}.

\begin{theorem}\label{T3}\thmref{T3}
There is no universal $K_\sigma$ subgroup of $S_\infty^\omega$.
\end{theorem}

Theorem~\ref{T3} is, in essence, a ``sharpness'' result for Theorems~\ref{T1} and \ref{T2}, suggesting that neither can readily be improved.

The next two propositions will allow us to apply Theorems~\ref{T1} and \ref{T2} in a broader context, e.g., to groups which are not of the form $G^\NN$.

\begin{proposition}\label{P1}\propref{P1}
Suppose that $G_1$ and $G_2$ are topological groups such that there exist continuous injective homomorphisms $\varphi_1 : G_1 \rightarrow G_2$ and $\varphi_2 : G_2 \rightarrow G_1$.  Let $\mathcal C$ be a class of subgroups which is closed under continuous homomorphic images.  If $G_1$ has a universal $\mathcal C$ subgroup, then $G_2$ also has a universal $\mathcal C$ subgroup.
\end{proposition}

\begin{proof}
Let $H \subseteq G_1$ be a universal $\mathcal C$ subgroup of $G_1$.  We will show that $H' = \varphi_1(H)$ is a universal $\mathcal C$ subgroup of $G_2$.  Indeed, fix a $\mathcal C$ subgroup $K \subseteq G_2$.  By the closure properties of $\mathcal C$, $\varphi_2 (K)$ is a $\mathcal C$ subgroup of $G_1$.  Let $\psi : G_1 \rightarrow G_1$ be a continuous endomorphism such that $\varphi_2 (K) = \psi^{-1} (H)$.  The injectivity of $\varphi_1$ and $\varphi_2$ implies that $K = (\varphi_1 \circ \psi \circ \varphi_2)^{-1} (H')$.  This completes the proof.
\end{proof}

\begin{remark*}
The hypothesis of Proposition~\ref{P1} is a very weak form of bi-embeddability, since the maps $\varphi_1$ and $\varphi_2$ need not have continuous inverses and hence need not be isomorphisms between their domains and ranges.
\end{remark*}

Since the class of $F_\sigma$ subgroups is not closed under continuous homomorphic images, we require a stronger form of mutual embeddability to apply Theorem~\ref{T2} to a larger class of groups.  The next proposition follows by the same proof as Proposition~\ref{P1}.

\begin{proposition}\label{P2}\propref{P2}
Suppose that $G_1$ and $G_2$ are Polish groups, each of which is isomorphic to a closed subgroup of the other.  If $G_1$ has a universal $F_\sigma$ subgroup for $K_\sigma$, then so does $G_2$.
\end{proposition}

As an example, consider the Polish group $\mathbf c_0^\omega$, where $\mathbf c_0$ is the additive group of sequences of real numbers tending to zero (with the sup-metric).  In the weak sense of Proposition~\ref{P1}, $\mathbf c_0^\omega$  is bi-embeddable with $\RR^\omega$.  (We give more details in Section~\ref{S7}.)  Proposition~\ref{P1} thus implies that $\mathbf c_0^\omega$ has universal $K_\sigma$ and universal compactly generated subgroups, even though $\mathbf c_0$ is not a product of locally compact groups.

Consider the groups $S_\infty$ and $\Aut (\QQ)$, both of which embed their countable powers as closed subgroups.  Theorem~\ref{T2} and Proposition~\ref{P2} together yield an $F_\sigma$ subgroup of $S_\infty$ (resp., of $\Aut (\QQ)$) which is universal for $K_\sigma$ subgroups of $S_\infty$ (resp., of $\Aut (\QQ)$).

In Section~\ref{S7}, we give more details of such examples and provide further applications of Theorem~\ref{T1} in conjunction with Proposition~\ref{P1}.  We will also show that a class of standard Banach and Hilbert spaces, viewed as complete topological groups, have universal $K_\sigma$ and compactly generated subgroups.

Section~\ref{S4.1} is a brief detour exploring the relationship between $K_\sigma$ and compactly generated subgroups.  We obtain the following result.

\begin{theorem}\label{T4}\thmref{T4}
Suppose that $G$ is countable discrete group.  Every $K_\sigma$ subgroup of $G^\omega$ is compactly generated if, and only if, every subgroup of $G$ is finitely generated.
\end{theorem}

In particular, every $K_\sigma$ subgroup of the countable power of a finite group is compactly generated.  Likewise, in $\mathbb Z^\omega$.  Regarding this result, Arnold Miller \cite{MILLER.countable.groups} has shown in  that every $K_\sigma$ subgroup of $\mathbb R^\omega$ is compactly generated.  This is an interesting complement to the theorem above as $\mathbb R$ is not discrete and, in general, countable subgroups of $\mathbb R$ are not finitely generated.

In Section~\ref{S8}, we apply the method of Theorem~\ref{T1} to demonstrate the existence of complete $F_\sigma$ ideals on $\omega$, with respect to a weak form of Rudin-Keisler reduction.



\section{Preliminaries and notation}\label{S2}

The definitions and notation we use are standard and essentially identical to those in Kechris \cite{KECHRIS.dst} and Kanovei \cite{KAN.borel.equiv}.  We recall some key points below.

A {\it Polish space} is a separable space whose topology is compatible with a complete metric.  A {\it topological group} is a topological space $G$ equipped with a group operation and an inverse map, such that the group operation is continuous as a function $G^2 \rightarrow G$ and the inverse map is continuous as a function $G \rightarrow G$.  Hence a {\it Polish group} is a topological group, the topology of which is Polish.

Except when working with specific groups, we will always use multiplicative notation for group operations.  Henceforth, $\mathbf 1$ will denote the identity element of a multiplicative group.

It is useful to have the notion of a {\it group word}.  An $n$-ary group word $w$ is a function taking $n$ symbols as input and combining these symbols using multiplication and inverses.  For example, $w (a,b,c) = b^{-1}ac^{-1}$ is a ternary group word.  For an $n$-ary group word $w$ and a topological group $G$, note that $w$ induces a continuous function $G^n \rightarrow G$.  When there is no ambiguity, we will sometimes write $w$ for $w(a_1 , \ldots , a_n)$.

For $A \subseteq G$, we let $w [A]$ denote the set 
\[
\{ w (x_1 , \ldots , x_m ) : x_1 , \ldots , x_m \in A\}
\]
We let $\langle A \rangle$ denote the {\it subgroup generated by A}, i.e., the smallest (with respect to containment) subgroup of $G$ which contains $A$.  Equivalently, 
\[
\langle A \rangle = \bigcup \{ w [A] : w \mbox{ is a group word}\}.
\]

For subsets $A,B$ of a group $G$ and $g \in G$, we let $AB$ denote the set $\{ ab : a \in A \ \& \ b \in B\}$, $gA$ denote $\{ ga : a \in A\}$ and $A^{-1}$ denote $\{ a^{-1} : a \in A\}$.  Likewise, define $A+B$ and $-A$, in the case of additive groups.

If $x$ is any sequence, we let $x(n)$ denote the $n$th term (or {\it bit}) of $x$.  We denote the length $n$ initial segment of $x$ by $x \upto n$.  If $I \subset \omega$ is the interval $\{ k,k+1 , \ldots , k+m\}$, then $x \upto I$ denotes the finite sequence 
\[
(x(k) , x(k+1) , \ldots , x(k+m)).
\]
For a set $A$ of sequences, we let $A \upto n$ denote the set $\{ x \upto n : x \in A\}$.

For finite sequences $s,t$, $s \concat t$ denotes the concatenation of $s$ and $t$.  If $t$ is the length 1 sequence $(a)$, for some $a \in X$, we simply write $s \concat a$, for $s \concat t$.

If $X$ is any set and $a \in X$, $a^n$ denotes the finite sequence $(a , \ldots ,a ) \in X^n$ and $\bar a$ the infinite sequence $(a,a, \ldots ) \in X^\omega$.

If $T \subseteq X^{<\omega}$ is a tree, then $[T]$ denotes the set $\{ x \in X^\omega : (\forall n) (x \upto n \in T)\}$.  If $T \subseteq (X \times Y)^{< \omega}$ is a tree, then $p[T]$ denotes the set $\{ x \in X^\omega : (\exists y \in Y^\omega) ((x,y) \in [T] )\}$.

For $\alpha, \beta \in \omega^\omega$, we write $\alpha \leq \beta$ to mean that $(\forall i) (\alpha (i) \leq \beta (i))$.  Similarly, if $s,t \in \omega^k$, $s \leq t$ means that $s(i) \leq t(i)$, for each $i < k$.

Finally, if $A$ is a subset of a topological space $X$, $\cl ( A)$ denotes the (topological) closure of $A$.



\section{A universal closed subgroup of $\mathbb Z^\omega$}\label{S3}

The following is our simplest result.  Although it does not fit into the scheme outlined in Section~\ref{summary}, it provides an example of the type of ``coding'' we will use to produce universal subgroups.

\begin{theorem}\label{closed}
There is a universal closed subgroup of $\mathbb Z^\omega$.
\end{theorem}

\begin{proof}
$\mathbb Z^k$ is a free, finitely-generated, Abelian group.  Hence all of its subgroup are also finitely generated (see Lang \cite[7.3]{LANG.algebra}).  In particular, there are only countably many subgroups of $\mathbb Z^k$.  Enumerate them as $G^k_0, G^k_1, \ldots$.  For each $n,k$, let $I^k_n \subseteq \omega$ be an interval of length $k$, such that that $\{ I^k_n : n,k \in \omega \}$ partition $\omega$.  

Define a closed subgroup $G$ of $\mathbb Z^\omega$ by  
\[ 
x \in G \iff (\forall k,n) (x \upto I^k_n \in G^k_n).
\]
We will show that $G$ is a universal closed subgroup.  Let $H$ be an arbitrary closed subgroup of $\mathbb Z^\omega$.  We show that $H \leqg G$.

Let $T$ be a pruned tree on $\mathbb Z$ such that $H = [T]$.  Note that, because $T$ is pruned, $T \cap \mathbb Z^k$ is a subgroup of $\mathbb Z^k$, for each $k$.  Given $k$, let $n_k$ be such that $T \cap \mathbb Z^k = G^k_{n_k}$.

Define a continuous group homomorphism $\varphi : \mathbb Z^\omega \rightarrow \mathbb Z^\omega$ by
\[ 
\varphi(x) \upto I^k_n = 
\begin{cases}
x\upto k \quad &\mbox{if} \quad n = n_k,\\
0^k  &\mbox{otherwise.}
\end{cases}
\]
For $x \in \mathbb Z^\omega$ and $y = \varphi (x)$, we have 
\begin{align*}
x \in H &\iff (\forall k) (x\upto k \in T \cap \mathbb Z^k) \\
&\iff (\forall k) (y \upto I^k_{n_k} \in G^k_{n_k}) \\
&\iff (\forall k,n) (y \upto I^k_n \in G^k_n)\\
&\iff \varphi (x) \in G.
\end{align*} 
The third `$\iff$' follows from the fact that, if $n \neq n_k$, then $y \upto I^k_n = 0^k \in G^k_n$.  This shows that $H \leqg G$.
\end{proof}

If $G$ is a finite group, then there are only finitely many subgroups of $G^k$, for each $k$.  Thus we have the following corollary to the proof of Theorem~\ref{closed}.

\begin{corollary}
If $G$ s a finite group, then $G^\omega$ has a universal closed subgroup.
\end{corollary}



\section{$K_\sigma$ subgroups}\label{S4}

In Section~\ref{S4.1} we study the relationship between $K_\sigma$ and compactly generated subgroups and prove Theorem~\ref{T4}.  In Section~\ref{S4.2}, we produce universal $K_\sigma$ and compactly generated subgroups in the direct product of any sequence of locally compact Polish groups, with infinitely often repeated factors, thereby proving Theorem~\ref{T1}.

\subsection{$K_\sigma$ vs.~compactly generated subgroups}\label{S4.1}

A compactly generated subgroup will always be $K_\sigma$.  Examples of such subgroups in $\mathbb Z^\omega$ are 
\[
\mathcal B = \{ x : x \mbox{ is bounded}\}
\]
(generated by the set of all 0-1 sequences) and 
\[
\mathrm{Fin} = \{ x : (\forall^\infty n) (x(n) = 0) \}
\]
(generated by the set of 0-1 sequences with at most one nonzero bit).  

In some cases, the classes of $K_\sigma$ and compactly generated subgroups coincide.  The following two theorems give a sufficient condition for this to be the case.  In particular, they imply that every $K_\sigma$ subgroup of $\mathbb Z^\omega$ is compactly generated.

\begin{theorem}
For a Polish group $G$, every $K_\sigma$ subgroup of $G$ is compactly generated if, and only if, every countable subgroup of $G$ is compactly generated.\footnote{For countable subgroups, note that compactly generated is not the same as finitely generated, e.g., $\mathbb Q \subseteq \mathbb R$ is generated by $\{ {1\over n } : n \in \omega \} \cup \{ 0 \}$, but is not finitely generated.}
\end{theorem}

\begin{proof}
The ``only if'' part follows from the fact that every countable subgroup is $K_\sigma$.

For the ``if'' part, suppose that $H = \bigcup_n K_n$ is a $K_\sigma$ subgroup of $G$.  Let $U_0 \supseteq U_1 \supseteq \ldots $ be a neighborhood base at the identity element $\mathbf 1 \in G$, with the additional property that each $\cl (U_{n+1}) \subseteq U_n$.  For each $n$
\[
\{ xU_{n+1} : x\in K_n \}
\]
covers $K_n$.  By compactness, there exists a finite set $S_n \subseteq K_n$ such that \[\{ x U_{n+1} : x \in S_n \}\] still covers $K_n$.  Now let 
\[
K^\ast_n = \bigcup_{x\in S_n} x^{-1}((x \cl (U_{n+1})) \cap K_n).
\]
First note that, as the finite union of translates of compact sets, $K^\ast_n$ is compact.  Also, $K^\ast_n \subseteq H$ and $\mathbf 1 \in K^\ast_n \subseteq \bar U_{n+1} \subseteq U_n$.  Furthermore, $K_n \subseteq \langle K^\ast_n \cup S_n\rangle$.  Let $K^\ast = \bigcup_n K^\ast_n$.  Then $K^\ast \subseteq H$ and 
\[ 
H = \langle K^\ast \cup \bigcup_n S_n\rangle.
\]

We claim that $K^\ast$ is compact.  Indeed, suppose that $z_0 , z_1 , \ldots \in K^\ast$.  If there is $n$ such that $z_j \in K^\ast_n$, for infinitely many $j$, then $(z_j)_{j \in \omega}$ has a subsequential limit in $K^\ast_n$, by compactness.  On the other hand, suppose that there are only finitely many $z_j$ in each $K^\ast_n$.  Let $n_0 < n_1 < \ldots$ and $j_0 , j_1 , \ldots$ be such that for each $k$, $z_{j_k} \in K^\ast_{n_k}$.  Then for each $k$, $z_{j_k} \in U_{n_k}$.  Hence $z_{j_k} \rightarrow \mathbf 1 \in K^\ast$, as $k \rightarrow \infty$.

Let $S \subseteq H$ be the subgroup generated by $\bigcup_n S_n$ (a countable subgroup).  By assumption, $S$ is compactly generated.  Therefore, take a compact set $C \subseteq S$ with $S = \langle C \rangle$.  Then $H$ will be generated by the compact set $K^\ast \cup C$.
\end{proof}

We now restate and prove Theorem~\ref{T4}.

\begin{refthmT4}
Suppose that $G$ is countable discrete group.  Every $K_\sigma$ subgroup of $G^\omega$ is compactly generated if, and only if, every subgroup of $G$ is finitely generated.
\end{refthmT4}

\begin{proof}
First suppose that there is a subgroup $H$ of $G$ which is not finitely generated.  Then $H^* = \{ \bar a : a \in H\}$ is a $K_\sigma$ subgroup of $G^\omega$ with no compact generating set.

Suppose now that every subgroup of $G$ is finitely generated.  We will show that every $K_\sigma$ subgroup of $G^\omega$ is compactly generated.  By the previous theorem, it suffices to show that every countable subgroup of $G^\omega$ is compactly generated.  Fix a countable subgroup $C = \{ x_1 , x_2 , \ldots \}$.  For each $n$, let $C_n = \{ x \in C : x \upto n = \mathbf 1^n\}$.  

\begin{claim1}
For each $n$, there is a finite set $F_n \subseteq C_n$ such that if $x \in C_n$, then there exists a group word $w$ in the elements of $F_n$ such that $x \cdot w^{-1} \in C_{n+1}$.
\end{claim1}

\proofclaim  For each $C_n$ there is a finite set $F_n \subseteq C_n$ such that $\{ x(n) : x \in F_n \}$ generates $\{ x(n) : x \in C_n\}$, since the latter is a subgroup of $G$.  

This implies that, for each $x \in C_n$ there is a group word $w$ in the elements of $F_n$ such that $x(n) = w (n)$.  Hence $x(n) \cdot w^{-1} (n) = \mathbf 1$.  On the other hand, $x \upto n = w \upto n = \mathbf 1^n$, since $x , w \in C_n$.  Thus 
\[
x\cdot w^{-1} \upto (n+1) = \mathbf 1^{n+1}.
\]
In other words, $x \cdot w^{-1} \in C_{n+1}$.  This proves the claim.

\begin{claim2}
For each $n$ there exists $\tilde x_n \in C_n$ and a group word $w_n$ in the elements of $F_0 \cup \ldots \cup F_{n-1}$ such that $x_n = \tilde x_n \cdot w_n$.  
\end{claim2}

\proofclaim  The argument is a finite induction.  Let $w_{n,0}$ be a group word in the elements of $F_0$, as in Claim 1, such that $x_n \cdot w_{n,0}^{-1} \in C_1$.  Set $x_{n,1} = x_n \cdot w_{n,0}^{-1}$.  Now let $w_{n,1}$ be a group word in the elements of $F_1$ such that $x_{n,1} \cdot w_{n,1}^{-1} \in C_2$ and define $x_{n,2} = x_{n,1} \cdot w_{n,1}^{-1}$.  In general, we obtain $x_{n,i} \in C_i$ and group words $w_{n,i}$ in the elements of $F_i$ such $x_{n,i+1} = x_{n,i} \cdot w_{n,i}^{-1} \in C_{i+1}$.

Let $\tilde x_n = x_{n,n}$ and $w_n = w_{n,n-1} \cdot \ldots \cdot w_{n,0}$.  Observe that $w_n$ is a group word in the elements of $F_0 \cup \ldots \cup F_{n-1}$, $\tilde x_n \in C_n$ and $x_n = \tilde x_n \cdot w_n$, as desired.

\medskip

Claim 2 implies that each $x_n$ is in the subgroup generated by $\tilde x_n$ together with $F_0 \cup \ldots \cup F_{n-1}$.  Thus the set 
\[
\tilde C = \bigcup_n (\{ \tilde x_n \} \cup F_n)
\]
generates $C$. 

It remains to check that $\tilde C$ is compact.  For each $n$, observe that there are only finitely many elements $x \in \tilde C$ such that $x(n) \neq \mathbf 1$, since all such elements are contained in $\{\tilde x_i : i \leq n \} \cup F_0 \cup \ldots \cup F_n$.  Thus every infinite sequence of distinct elements of $\tilde C$ must converge to $\bar{\mathbf 1}$.  This implies that every infinite sequence in $\tilde C$ is either eventually constant or has a subsequence converging to $\bar{\mathbf 1}$.
\end{proof}

We enumerate a couple of direct consequences.
\begin{enumerate}
\item Every $K_\sigma$ subgroup of $\mathbb Z^\omega$ is compactly generated.  (Since every subgroup of $\mathbb Z$ is singly generated.)
\item If $G$ is a finite group, then every $K_\sigma$ subgroup of $G^\omega$ is compactly generated.
\end{enumerate}

For a Polish group $G$, even if there are non-compactly generated $K_\sigma$ subgroups, we can still ask whether or not every $K_\sigma$ subgroup is group-homomorphism reducible to a compactly generated one.  The following two examples illustrate the range of possibilities.

\begin{example} Let $S = \bigoplus_\omega \mathbb Z$ be the direct sum of countably many copies of $\mathbb Z$.  Unlike $\mathbb Z$, the countable group $S$ is not finitely generated.  Thus, with the discrete topology, $S$ is $K_\sigma$, but not compactly generated.  (In a discrete space, compact is the same as finite.)

By extension, not all $K_\sigma$ subgroups of $S^\omega$ will be compactly generated.  For example, $\{ x \in S^\omega : x \mbox{ is a constant sequence}\}$.  On the other hand, we will see that every $K_\sigma$ subgroup is group-homomorphism reducible to a compactly generated one.  We begin by showing that $S^\omega$ may be mapped one-to-one homomorphically into $\mathbb Z^\omega$.  Let $\varphi_n : S \rightarrow \mathbb Z$ be the projection map onto the $n$th coordinate.  Define $\psi : S^\omega \rightarrow \mathbb Z^\omega$ by 
\[
\psi (x)(\langle m,n \rangle) = \varphi_n(x(m)),
\]
where $\langle \cdot , \cdot \rangle: \omega^2 \longleftrightarrow \omega$ is a fixed bijection.  The map $\psi$ is a continuous injective homomorphism whose range is the $\mathbf \Pi^0_3$ subgroup
\[
\{ y \in \mathbb Z^\omega : (\forall m) (\forall^\infty n) ((x (\langle m,n \rangle) = 0)\}.
\]

Now let $H \subseteq S^\omega$ be any $K_\sigma$ subgroup.  The image $\psi(H) \subseteq \mathbb Z^\omega$ is also $K_\sigma$ (because $\psi$ is continuous) hence compactly generated by Theorem~\ref{T4}.  Say $\psi (H) = \langle K \rangle$.  Let $i : \mathbb Z^\omega \rightarrow S^\omega$ be the natural ``inclusion'' map.  Then $i (K) \subseteq S^\omega$ is compact and $H = (i\circ \psi)^{-1} (\langle i(K) \rangle)$, because $i \circ \psi$ is injective.  
\end{example}

For our next example, we introduce some terminology.  Suppose that $H$ is a subgroup of an Abelian group $G$ (with additive notation) and $x \in H$.  We say that $x$ is \emph{divisible in $H$} to mean that for each $n \in \omega$, there exists $y \in H$ such that $x = ny$.  Note that for subgroups $H_1, H_2 \subseteq G$, if $\varphi : G \rightarrow G$ is a group homomorphism such that $\varphi^{-1} (H_2) = H_1$ and $x \in H_1$ is divisible in $H_1$, then $\varphi (x) \in H_2$ is divisible in $H_2$.

\begin{example} Consider the group $\mathbb Q$ of rational numbers with the discrete topology.  We will see that there are $K_\sigma$ subgroups of $\mathbb Q^\omega$ that are not group-homomorphism reducible to any compactly generated subgroup.

We first claim that there are no nonzero divisible elements in a compactly generated subgroup of $\mathbb Q^\omega$.  Indeed, suppose that, on the contrary, $H$ is generated by the compact set $K$ and there is a nonzero element $x \in H$, with $x$ divisible in $H$.  Let $m \in \omega$ be such that $x(m) \neq 0$.  Let 
\[ 
A = \{ y(m) : y \in K\}.
\]
Note that, since $x$ is divisible in $H$, $x(m)$ will be divisible in $\langle A \rangle \subset \mathbb Q$.  As $K$ is compact and we have given $\mathbb Q$ the discrete topology, $A$ must be finite.  Therefore, let $k \in \mathbb Z$ be such that $ka \in \mathbb Z$, for each $a \in A$.  This implies that, for any $b \in \langle A \rangle$, we also have $kb \in \mathbb Z$.  Let $n$ be large enough that ${k \over n}x(m) \notin \mathbb Z$.  Thus ${1\over n}x(m) \notin \langle A \rangle$, contradicting the divisibility of $x(m)$ in $\langle A \rangle$.

We now exhibit a $K_\sigma$ subgroup which is not group-homomorphism reducible to any compactly generated subgroup.  Consider the subgroup 
\[
\mathrm{Fin} = \{ x \in \mathbb Q^\omega : (\forall^\infty n) (x(n) = 0)\}.
\]
$\mathrm{Fin}$ is $K_\sigma$ and every element of $\mathrm{Fin}$ is divisible in $\mathrm{Fin}$.  Suppose that $\varphi : \mathbb Q^\omega \rightarrow \mathbb Q^\omega$ is a continuous homomorphism and $H$ is a subgroup such $\varphi^{-1}(H) = \mathrm{Fin}$.  In the first place, we have that $\ker (\varphi) \subseteq \mathrm{Fin}$.  Note, however, that $\ker(\varphi) \neq \mathrm{Fin}$, since then we would have $\varphi \equiv 0$ because $\mathrm{Fin}$ is dense in $\mathbb Q^\omega$.  Hence there exists $x \in \mathrm{Fin}$ with $\varphi(x) \neq 0$.  Since $x$ is divisible in $\mathrm{Fin}$, we have that $\varphi (x)$ is divisible in $H$ and nonzero.  Thus $H$ cannot be compactly generated, by the comments above.
\end{example}

\subsection{Universal subgroups}\label{S4.2}

Before giving the proof of Theorem~\ref{T1}, we consider the special case of $\ZZ^\omega$ to illustrate the main ideas in a more straightforward setting.

\subsubsection{The case of $\mathbb Z^\omega$}

\begin{theorem}\label{ucg}
There is a universal compactly generated subgroup of $\mathbb Z^\omega$.
\end{theorem}

\begin{proof}
We essentially construct a $\leqg$-complete compact subset of $\mathbb Z^\omega$.

For each $m \in \omega$, let $A^k_0, A^k_1, \ldots$ list all finite subsets of $\mathbb Z^k$ which contain $0^k$ and are such that $-A^k_j = A^k_j$.  Let $I^k_j$ ($k,j \in \omega$) partition $\omega$, with each $I^k_j$ an interval of length $k$.  Define $K_0 \subset \mathbb Z^\omega$ by
\[ 
x\in K_0 \iff (\forall k,j) (x\upto I^k_j \in A^k_j).
\]
Note that $K_0$ is compact and $-K_0 = K_0$.  Consider $\langle K_0 \rangle$ (the subgroup generated by $K_0$).  We show that $\langle K_0 \rangle$ is universal for compactly generated subgroups of $\mathbb Z^\omega$. 

Suppose that $\langle K \rangle$ is any compactly generated subgroup.  With no loss of generality, we assume that $-K = K$ and $\bar 0 \in K$.  There is a pruned tree $T$ on $\mathbb Z$ such that $K = [T]$.  Since $K$ is compact, all levels of $T$ must be finite.  For each $k$, choose $\tau(k) \in \omega$ such that $A^k_{\tau(k)} = T \cap \mathbb Z^k$.  Define a homomorphism $\varphi: \mathbb Z^\omega \rightarrow \mathbb Z^\omega$ by 
\[
\varphi(x) \upto I^k_j =
\begin{cases}
x\upto k \quad &\mbox{if $j = \tau(k)$,}\\ 
 0^k &\mbox{otherwise.}
\end{cases}
\]
Observe that $\varphi^{-1}(K_0) = K$.  The following claim will complete the proof of this theorem.

\begin{claim}
$\varphi^{-1} (\langle K_0 \rangle) = \langle K \rangle$.
\end{claim}

\proofclaim  Suppose that $x \in \langle K \rangle$, with $x_1, \ldots , x_m \in K$ such that $x = x_1 + \ldots + x_m$.  (Note that, since $-K = K$, all elements of $\langle K \rangle$ are finite sums of elements of $K$.)  Then $\varphi (x_1) , \ldots , \varphi (x_m) \in K_0$ and hence $\varphi (x) = \varphi (x_1) + \ldots + \varphi (x_m) \in \langle K_0 \rangle$.

Suppose, on the other hand, that $\varphi (x) \in \langle K_0 \rangle$, with $y_1, \ldots , y_m \in K_0$ such that $\varphi (x) = y_1 + \ldots + y_m$.  (Again, because $-K_0 = K_0$, $\langle K_0 \rangle$ is the set of finite sums of members of $K_0$.)  We want $x_1 , \ldots , x_m \in K$ with $x = x_1 + \ldots + x_m$. 

For each $i\leq m$, let $v^k_i = y_i \upto I^k_{\tau(k)}$.  Since each $y_i \in K_0$, the definition of $K_0$ implies that each 
\[
v^k_i \in A^k_{\tau(k)} = T \cap \mathbb Z^k.
\]
Hence (because $T$ is pruned) there exists $x^k_i \in K$ such that 
\[
x^k_i \upto k = v^k_i.
\]

By the compactness of $K$, we may iteratively (for $i \leq m$) take convergent subsequences of $(x^k_i)_{k\in \omega}$ to obtain a common subsequence $k_0 < k_1 < \ldots$ such that, for each $i \leq m$, $(x^{k_n}_i)_{n\in \omega}$ is convergent, with limit $x_i \in K$.  Finally, fix $p$ and let $k_n \geq p$ be large enough that $x_i^{k_n} \upto p = x_i \upto p$, for each $i \leq m$.  Thus
\begin{align*}
x \upto p
&= \sum_{i \leq p} v^{k_n}_i \upto p\\
&= \sum_{i \leq p} x^{k_n}_i \upto p \quad (\mbox{because } k_n \geq p)\\
&= \sum_{i \leq p} x_i \upto p.
\end{align*}
As $p$ was arbitrary, we have $x = \sum_{i \leq m} x_i \in \langle K \rangle$.  This completes the proof.
\end{proof}

\begin{corollary}\label{uk}
There is a universal $K_\sigma$ subgroup of $\mathbb Z^\omega$.
\end{corollary}

\begin{proof}
Since every $K_\sigma$ subgroup of $\mathbb Z^\omega$ is compactly generated by Theorem~\ref{T4}, Theorem~\ref{ucg} actually gives a universal $K_\sigma$ subgroup of $\mathbb Z^\omega$.
\end{proof}

\subsubsection{Proof of Theorem~\ref{T1}}

In this section, we prove the following theorem.  

\begin{refthmT1}
Let $(G_n)_{n \in \omega}$ be a sequence of locally compact Polish groups, each term of which occurs infinitely often (up to isomorphism).  We have the following;
\begin{enumerate}
\item $\prod_n G_n$ has a universal compactly generated subgroup.
\item $\prod_n G_n$ has a universal $K_\sigma$ subgroup.
\end{enumerate}
\end{refthmT1}

Note that if every $K_\sigma$ subgroup of $\prod_n G_n$ is reducible to a compactly generated subgroup, then part (1) of this result implies part (2).  On the other hand, in Section~\ref{S4.1} we saw examples of $K_\sigma$ subgroups of Polish groups which do not reduce to compactly generated subgroups.  In such cases, (1) and (2) remain distinct results.

As mentioned in the Introduction, the following is a corollary of Theorem~\ref{T1}.

\begin{corollary}\label{k power}
If $G$ is locally compact, then $G^\omega$ has universal compactly generated and $K_\sigma$ subgroups.
\end{corollary}

For most of the examples we consider in Section~\ref{S7}, we will only use the statement of Corollary~\ref{k power}.  

Our key lemma in the proof of Theorem~\ref{T1} is a restricted, but refined, version of Theorem~\ref{T1}(1).  (Recall that for an $m$-ary group word $w$, we define $w[K] = \{ w (x_1 , \ldots , x_m) : x_1 , \ldots , x_m \in A\}$.)

\begin{lemma}\label{loki}
Let $G$ be a locally compact Polish group with identity element $\mathbf 1$.  There exists a compact set $K_0 \subseteq G^\omega$ with $\bar{\bf 1} \in K_0$ and the property that for each compact $K \subseteq G^\omega$, with $\bar{\bf 1} \in K$, there is a continuous group homomorphism $\varphi : G^\omega \rightarrow G^\omega$ such that, for each group word $w$,
\[
\varphi^{-1} (w[K_0]) = w [K].
\]
In particular, $\langle K_0 \rangle$ is a universal compactly generated subgroup of $G^\omega$.
\end{lemma}

\subsubsection*{Basic notions}

We begin with some notation and facts we will use in the proof of Lemma~\ref{loki}.  From now on, fix a locally compact Polish group $G$, with identity element $\mathbf 1$.

The following lemma gives a neighborhood base at $\mathbf 1$ with the specific properties we require.

\begin{lemma}\label{nbhd base}
There is a neighborhood base $\{ U_k \}$ at $\mathbf 1$ such that
\begin{enumerate}
\item Each $U_k$ has compact closure.
\item $U_0 \supseteq U_1 \supseteq \ldots$.
\item For each $k$, $U_k^{-1} = U_k$.
\item For each $k > 0$, $\cl (U_k U_k)  \subseteq U_{k-1}$.
\end{enumerate}
\end{lemma}

\begin{proof}

We construct the $U_k$ inductively.  Let $V_0 \supseteq V_1 \supseteq \ldots \ni \mathbf 1$ be any ``nested'' neighborhood base at $\mathbf 1$, such that $\cl (V_0)$ is compact.  (Such $V_k$ exist since $G$ is locally compact.)  Let $U_0 = V_0$.  Suppose that $U_0 \supseteq \ldots \supseteq U_k$ are given with the desired properties.  By the continuity of the group operation, there is a neighborhood $V$ of $\mathbf 1$ such that $\cl (V V) \subseteq V_k \cap U_k$.  By the continuity of the map $(x,y) \mapsto x^{-1}y$, there is a neighborhood $W$ of $\mathbf 1$ such that $W^{-1} W \subseteq V$.  Let $U_{k+1} = W^{-1}W$.  Then $(U_{k+1})^{-1} = U_{k+1}$ and 
\[ 
\cl (U_{k+1}U_{k+1}) \subseteq \cl (VV) \subseteq U_k.
\]
\end{proof}

Fix a neighborhood base $\{ U_k \}$, as in the lemma above.  For $a,b \in G$, write $a \approx_k b$ (``$a$ $k$-approximates $b$'') if, and only if, $a^{-1}b \in \cl ( U_k )$.  Note that, by the properties of the $U_k$, 
\begin{enumerate}
\item $a \approx_k a$
\item $a \approx_k b \iff b \approx_k a$
\item $a \approx_k b \approx_k c \implies a \approx_{k-1} c$
\item $(a \approx_k b \ \& \ k' \leq k) \implies a \approx_{k'} b$
\item $\lim_n a_n = a \iff (\forall k)(\forall^\infty n ) (a_n \approx_k a)$.
\end{enumerate}
If $x,y \in G^\omega$ (or $G^p$), we will write $x \approx_k y$ to indicate that $x(i) \approx_k y(i)$, for each $i \in \omega$ (or $i < p$).  Item 5 above implies that for $x, x_n \in G^\omega$
\begin{equation}\label{E2}
\lim_n x_n = x \iff (\forall p,k)(\forall^\infty n) (x_n \upto p \approx_k x \upto p) 
\end{equation}
Also note that, for each $k$ and fixed $a_0 \in G$, the set
\[\{ a \in G : a_0 \approx_k a \}\]
is compact.

Fix a countable dense set $D \subseteq G$, with $\mathbf 1 \in D$.  Let $\mathfrak n \leq \omega$ be the cardinality of $D$, and $\# : D \longleftrightarrow \mathfrak n$ be a bijection, with $\# \mathbf 1 = 0$.

For $x \in G^\omega$ and $k \in \omega$, we define a sequence $\beta_x^k \in D^\omega$ (which we call the {\it least $k$-approximation} of $x$) as follows; for each $i$, let $a_i \in D$ be the element with $\# a_i$ least such that $a_i \approx_k x(i)$.  Define $\beta_x^k \in D^\omega$ by 
\[ 
(\forall i) (\beta_x^k (i) = a_i).
\]
Given a closed set $K \subseteq G^\omega$ and $k \in \omega$, let 
\[ 
\mathcal B_k = \{ \beta_x^k \upto k : x \in K \}.
\]
Since $K$ is closed, \eqref{E2} above implies that $x \in K$ if, and only if, $(\forall k) (\beta_x^k \upto k \in \mathcal B_k)$.  We have the following fact.

\begin{lemma}\label{bounded}
If $K \subseteq G^\omega$ is compact, then $\{ \beta_x^k (n) : x \in K\}$ is finite, for each $k,n\in \omega$.  In particular, each $\mathcal B_k$ is finite.
\end{lemma}

\begin{proof}
Since $K$ is compact, so is the set $A =  \{ x(n) : x \in K \} \subseteq G$.   There is thus a finite set $F_n \subseteq D$ such that, for each $x \in K$, there is an $a \in F_n$ with $x(n) \approx_k a$.  As $\beta_x^k (n)$ is the \#-least element of $D$ which $k$-approximates $x(n)$, we conclude that $\# \beta_x^k (n) \leq \max \{ \# a : a \in F_n \}$, for each $x \in K$.  Hence $\{ \beta_x^k (n) : x \in K \}$ is finite.

This implies that each $\mathcal B_k$ is finite, since $\mathcal B_k \subseteq \prod_{n<k} F_n$.
\end{proof}

\subsubsection*{Proof of Lemma~\ref{loki}}

Fix a locally compact group $G$ and let $D$, \#, $\approx_k$ be defined as above for $G$.  For each $k \in \omega$, let $A^k_0 , A^k_1 , \ldots \subseteq D^k$ be such that, for each $k,j$, we have
\begin{itemize}
\item $A^k_j$ is finite.
\item $\mathbf 1^k \in A^k_j$.
\item For each finite $A \subseteq D^k$, with $\mathbf 1^k \in A$, there exists $j$ such that $A = A^k_j$.
\end{itemize}
Let $I^k_j$ (for $k,j \in \omega$) be intervals partitioning $\omega$ such that each $I^k_j$ has length $k$.  Define $K_0 \subseteq G^\omega$ by 
\[
x \in K_0 \iff (\forall k,j) (\exists u \in A^k_j) (u \approx_k x \upto I^k_j).
\]
Note that $K_0$ is compact since ``$u \approx_k x \upto I^k_j\,$'' defines a compact subset of $G^k$ and the existential quantifier is over a finite set.  We shall show that $\langle K_0 \rangle$ has the property that for any compact $K \subseteq G^\omega$, containing $\bar 1$, there is a continuous homomorphism $\varphi : G^\omega \rightarrow G^\omega$ with
\[
\varphi^{-1} (w [K_0]) = w [K],
\]
for each group word $w$.

Let $K$ be an arbitrary compact subset of $G^\omega$, with $\bar{\bf 1}\in K$.  For each $k$, let
\[
\mathcal B_k = \{ \beta_x^k \upto k : x \in K \}
\]
be as above.  As we remarked in Lemma~\ref{bounded}, the compactness of $K$ implies that each $\mathcal B_k$ is finite.  Since $\bar{\bf 1}$ is its own least $k$-approximation, each $\mathcal B_k$ contains $\mathbf 1^k$.  For each $k \in \omega$, we may therefore choose $\tau(k) \in \omega$ such that 
\[
A^k_{\tau(k)} = \mathcal B_k. 
\]
Define a continuous group homomorphism $\varphi : G^\omega \rightarrow G^\omega$ by 
\[
\varphi(x) \upto I^k_j =
\begin{cases}
x \upto k   \quad&\mbox{if } j = \tau(k),\\
\mathbf 1^k  &\mbox{otherwise.}\\
\end{cases}
\]
Fix an $m$-ary group word $w$.  The following two claims will complete the proof.

\begin{claim1}
$x \in w [K] \implies \varphi (x) \in w [K_0]$.
\end{claim1}

\proofclaim  Since $\varphi$ is a group homomorphism, it will suffice to show that $x \in K \implies \varphi (x) \in K_0$.  Suppose that $x \in K$.  For each $k$, let 
\[
u_k = \beta_x^k \upto k \in \mathcal B_k = A^k_{\tau(k)}.
\]
Hence $u_k \approx_k x \upto k = \varphi (x) \upto I^k_{\tau(k)}$.  On the other hand, if $j \neq \tau(k)$, then $\varphi(x) \upto I^k_j = \mathbf 1^k \in A^k_j$.  Putting these together, we see that 
\[ 
(\forall k,j)(\exists u \in A^k_j) (u \approx_k \varphi (x) \upto I^k_j).
\]
Thus $\varphi(x) \in K_0$.  This completes the claim.

\begin{claim2}
$\varphi(x) \in w [K_0] \implies x \in w [K]$.
\end{claim2}

\proofclaim  Let $y_1 , \ldots , y_m \in K_0$ be such that $\varphi(x) = w (y_1 , \ldots , y_m)$.  We will find $x_1 , \ldots , x_m \in K$ such that $x = w (x_1 , \ldots , x_m)$ and conclude that $x \in w [K]$.

For each $k,i$, let 
\[
v^k_i = y_i \upto I^k_{\tau(k)} 
\]
and let $u^k_i \in A^k_{\tau(k)} = \mathcal B_k$ be such that $u^k_i \approx_k v^k_i$.  By the definition of $\mathcal B_k$, there exist $x^k_i \in K$ such that $u^k_i \approx_k x^k_i \upto k$, for each $k$ and $i \leq m$.  Since $K$ is compact, we may take $k_0 < k_1 < \dots$ and $x_1, \ldots , x_m \in K$ with $\lim_n x_i^{k_n} = x_i$, for each $i \leq m$.

Let $z^k_i = v^k_i \concat \bar{\bf 1}$.  We claim that
\[ 
\lim_n z^{k_n}_i = x_i. 
\]
Indeed, fix $p,r \in \omega$ and let $M$ be large enough that whenever $k_n \geq M$, we have 
\[ 
x^{k_n}_i \upto r \approx_{p+2} x_i \upto r.
\]
The existence of $M$ follows from \eqref{E2}, since $\lim_n x^{k_n}_i = x_i$.  We may assume that $M >r, p+2$ and so if $k_n \geq M$, we have
\begin{align*}
z^{k_n}_i \upto r 
&= v^{k_n}_i \upto r\\
&\approx_{p+2} u^{k_n}_i \upto r\\
&\approx_{p+2} x^{k_n}_i \upto r\\
&\approx_{p+2} x_i \upto r.
\end{align*}
Hence $z^{k_n}_i \upto r \approx_p x_i \upto r$, for all $k_n \geq M$.  As $p,r$ were arbitrary, we conclude (again, by \eqref{E2}) that $z^{k_n}_i \rightarrow x_i$ as $n \rightarrow \infty$.

We may now finish the claim.  Observe that for fixed $r$ and each $k_n > r$, we have 
\begin{align*}
x \upto r
&= (\varphi(x) \upto I^{k_n}_{\tau(k_n)}) \upto r\\
&= w (v^{k_n}_1 , \ldots , v^{k_n}_m) \upto r\\
&= w (z^{k_n}_1 , \ldots , z^{k_n}_m) \upto r.
\end{align*}
Taking the limit as $n \rightarrow \infty$ and using the fact that $w$ induces a continuous function $G^{rm} \rightarrow G^r$ we conclude that 
\[ 
x \upto r = w (x_1 , \ldots , x_m) \upto r. 
\]
Since $r$ was arbitrary, $x = w(x_1 , \ldots , x_m) \in w [K]$.  This completes the proof.\qed

\subsubsection*{Proof of main result}

We first prove (1) of Theorem~\ref{T1} and then prove (2) from (1).

\begin{proof}[Proof of Theorem~\ref{T1}(1)]

Let $(G_n)_{n \in \omega}$ be a sequence of locally compact Polish groups, with each term occuring infinitely often up to isomorphism.  This implies that $\prod_n G_n \cong \prod_n (G_0^\omega \times \ldots \times G_n^\omega) \cong \prod_n G_n^\omega$.  It will therefore suffice to show that there is a compactly generated subgroup of $\prod_n (G_0^\omega \times \ldots \times G_n^\omega)$ which is universal for compactly generated subgroups of $\prod_n G_n^\omega$.  

For each $n$, note that $G_0^\omega \times \ldots \times G_n^\omega \cong (G_0 \times \ldots \times G_n)^\omega$.  As the direct product of finitely many locally compact groups, $G_0 \times \ldots \times G_n$ itself is locally compact.  Therefore take compact sets $K_n \subseteq G_0^\omega \times \ldots \times G_n^\omega$ with $\bar{\bf 1} \in K_n$, as in Lemma~\ref{loki}, such that, for any compact $K \subseteq G_0^\omega \times \ldots \times G_n^\omega$ with $\bar{\bf 1} \in K$, there is a continuous endomorphism $\varphi$ of $G_0^\omega \times \ldots \times G_n^\omega$ such that $\varphi^{-1} (w[K_n]) = w [K]$, for each group word $w$.

Define a compact set $K_\infty \subseteq \prod_n (G_0^\omega \times \ldots \times G_n^\omega)$ by 
\[
\xi \in K_\infty \iff (\forall n) (\xi (n) \in K_n)
\]
We will show that $\langle K_\infty \rangle$ is universal for compactly generated subgroups of $\prod_n G_n^\omega$.  Indeed, fix an arbitrary compactly generated subgroup $\langle K \rangle \subseteq \prod_n G_n^\omega$.  We may assume that $\bar{\bf 1} \in K$.  For each $n$, Lemma~\ref{loki} gives an endomorphism $\varphi_n$ of $G_0^\omega \times \ldots \times G_n^\omega$ such that
\begin{equation}\label{E3}
\varphi_n^{-1} (w [K_n]) = w[K \upto (n+1)] = w [K] \upto (n+1),
\end{equation}
for each group word $w$.  (Recall here that $K\upto (n+1) = \{ x \upto (n+1) : x \in K\} \subseteq G_0^\omega \times \ldots \times G_n^\omega$.)  

Define a continuous homomorphism $\varphi : \prod_n G_n^\omega \rightarrow \prod_n (G_0^\omega \times \ldots \times G_n^\omega)$ by 
\[
\varphi (x) (n) = \varphi_n (x \upto (n+1)),
\]
for each $n$.  The following claim will complete the proof.

\begin{claim}
$\varphi^{-1} (\langle K_\infty \rangle) = \langle K \rangle$.
\end{claim}

\proofclaim  It suffices to show that, for each group word $w$, 
\begin{equation}\label{E4}
\varphi^{-1} (w [K_\infty]) = w[K].
\end{equation}
Fix a group word $w$.  Armed with \eqref{E3} and the fact that $w[K]$ is compact, we have
\begin{align*}
x \in w [K]
&\iff (\forall n) (x \upto (n+1) \in w [K] \upto (n+1))\\
&\iff (\forall n) (\varphi_n (x \upto (n+1)) \in w [K_n])\\
&\iff (\forall n) (\varphi (x) (n) \in w [K_n])\\
&\iff \varphi (x) \in w[K_\infty].
\end{align*}
The third ``$\iff$'' follows from the definition of $\varphi (x) (n)$ as $\varphi_n (x \upto (n+1))$.  This completes the proof.
\end{proof}

\begin{remark*}
In the proof above, \eqref{E4} and the definition of $K_\infty$ imply that the statement of Lemma~\ref{loki} holds for $\prod_n G_n$, i.e., $\bar{\bf 1} \in K_\infty$ and, for each compact $K\subseteq \prod_n G_n$ containing $\bar{\bf 1}$, there is a continuous homomorphism $\varphi : \prod_n G_n \rightarrow \prod_n G_n$ with $\varphi^{-1} (w [K_\infty]) = w [K]$, for each group word $w$.
\end{remark*}

Considering the group word $w_0 (a) = a$ and noting that $\langle K \rangle = \bigcup_{w} w[K]$, we obtain the following corollary to the proof of Theorem~\ref{T1}(1).

\begin{corollary}\label{loki product}
Suppose $(G_n)_{n \in \omega}$ are as above.  There exists a compact set $K_0 \subseteq \prod_n G_n$ such that $\bar{\bf 1} \in K_0$ and for each compact $K \subseteq \prod_n G_n$ with $\bar{\bf 1} \in K$, there is a continuous group homomorphism $\varphi : \prod_n G_n \rightarrow \prod_n G_n$ such that
\[
\varphi^{-1} (K_0) = K \ \ \mbox{ and }  \ \ \varphi^{-1}(\langle K_0 \rangle) = \langle K \rangle .
\]
\end{corollary}

We will use this in the next proof.

\begin{proof}[Proof of Theorem~\ref{T1}(2)]

Fix a sequence $(G_n)_{n \in \omega}$ of locally compact Polish groups, as above.  For each $n$, let $D_n \subseteq G_n$ be a countable dense set, containing the identity element $\mathbf 1_n \in G_n$.  For each $n$, fix an enumeration $\{ x^n_0 , x^n_1 , \ldots \}$ of $D_n$, with $x^n_0 = \mathbf 1_n$, and fix a neighborhood $U_n \ni \mathbf 1_n$, with $\cl (U_n)$ compact.

For each $n$ and $x \in \prod_n G_n$, define $x^\ast \in \omega^\omega$ by 
\[
x^\ast (n) = \min \{ i : (x^n_i)^{-1} x(n) \in \cl (U_n) \},
\]
for each $n \in \omega$.  Define $u^\ast \in \omega^n$ analogously, for $u \in \prod_{i < n} G_i$.  Observe that, by the argument of Lemma~\ref{bounded}, if $K \subseteq \prod_n G_n$ is compact, then $\{ x^\ast : x \in K\}$ has compact closure in $\omega^\omega$.  Conversely, since each $\cl (U_n)$ is compact, it follows that 
\[
\{ x \in \prod_n G_n : x^\ast \leq \alpha\}
\]
is compact, for each $\alpha \in \omega^\omega$.

For notational reasons, we will consider the group 
\[
G' = \prod_{\substack {n \in \omega \\ s\in \omega^{<\omega}}} (G_0 \times \ldots \times G_{\len s -1}).
\]
Note that $n$ is a ``dummy'' index, serving only to produce infinitely many copies of the term inside the product.  For the sake of clarity, we remark that $\xi (n,s) \in G_0 \times \ldots \times G_{\len s -1}$, for each $n,s$ and $\xi \in G'$.

Since each $G_n$ is isomorphic to infinitely many other $G_m$, we have $G' \cong \prod_n G_n$.  To prove our theorem, it will therefore suffice to produce a $K_\sigma$ subgroup of $G'$ which is universal for $K_\sigma$ subgroups of $\prod_n G_n$.

Let $K_0 \subseteq \prod_n G_n$ be as in Corollary~\ref{loki product}.  For each $n$, define
\[
A_n = \{ \xi \in G' : (\forall n' \geq n) (\forall s \in \omega^{< \omega}) (\xi (n',s) \in K_0 \upto \len s)\}.
\]
For each $n$, the subgroup $\langle A_n \rangle$ is $F_\sigma$.  This follows from the fact that each $A_n$ is the direct product of a compact set with factors of the form $G_0 \times \ldots \times G_k$.  

Define the set
\[
\tilde A = \{ \xi \in G' : (\forall^\infty n,s) (\xi (n,s)^\ast \leq s)\}.
\]
It follows that $\tilde A$ is $K_\sigma$ and hence $\langle \tilde A \rangle$ is as well.  Let
\[
H_0 = \langle \tilde A \rangle \cap \bigcup_n \langle A_n \rangle
\]
and note that, since the term $\bigcup_n \langle A_n \rangle$ is an increasing union of subgroups, $H_0$ itself is a subgroup of $G'$.  As the intersection of an $F_\sigma$ set with a $K_\sigma$ set, $H_0$ is $K_\sigma$.  We will show that $H_0$ is universal for $K_\sigma$ subgroups of $\prod_n G_n$.  

Let $B = \bigcup_n B_n$ be an arbitrary $K_\sigma$ subgroup of $\prod_n G_n$, with each $B_n$ compact and $\bar{\bf 1} \in B_0 \subseteq B_1 \subseteq \ldots$.  Take continuous endomorphisms $\psi_n$ of $\prod_n G_n$ such that 
\[
\psi_n^{-1} (K_0) = B_n \ \ \mbox{ and } \ \  \psi_n^{-1}(\langle K_0 \rangle) = \langle B_n \rangle ,
\]
for each $n$.  Each $\psi_m (B_n)$ is compact.  As noted above, this implies that the closure of $\{ x^\ast : x \in \psi_m(B_n)\}$ is compact in $\omega^\omega$.  Thus we may choose $\alpha_n \in \omega^\omega$ such that each $\alpha_n$ is increasing, $\alpha_0 \leq \alpha_1 \leq \ldots$ and $x^\ast \leq \alpha_n$, for each $x \in \bigcup_{n' \leq n} \psi_{n'} (B_n)$.  Define $\psi : \prod_n G_n \rightarrow G'$ by 
\[
\psi(x) (n,s) = \begin{cases}
\psi_n (x) \upto p &\mbox{if } s = \alpha_{n+p} \upto p,\\
(\mathbf 1_0 , \ldots , \mathbf 1_{p-1}) \quad&\mbox{otherwise,}
\end{cases}
\]
for each $n \in \omega$ and $s \in \omega^{<\omega}$ with $p = \len s$.  It remains to show that $\psi^{-1} (H_0) = B$.

\begin{claim1}
If $\psi(x) \in H_0$, then $x \in B$.
\end{claim1}

\proofclaim  Let $n$ be such that $\psi (x) \in \langle A_n \rangle$, with $w$ a group word such that $\psi (x) \in w[A_n]$.  For each $p$, if $s = \alpha_{n+p} \upto p$, we have 
\begin{align*}
\psi_n (x) \upto p
&= \psi (x) (n,s)\\
&\in \{ \xi (n,s) : \xi \in w [A_n]\}\\
&= w [K_0 \upto p]\\
&= w [K_0] \upto p
\end{align*}
and hence $\psi_n (x) \in w[K_0]$, since the latter is closed.  (As the continuous image of a compact set, $w[K_0]$ is compact.)  This implies that $\psi_n(x) \in \langle K_0 \rangle$ and, since $\psi_n$ reduces $\langle B_n \rangle$ to $\langle K_0 \rangle$, we conclude that $x \in \langle B_n \rangle \subseteq B$.

\begin{claim2}
If $x \in B$, then $\psi (x) \in H_0$.
\end{claim2}

\proofclaim  Suppose that $x \in B$, say $x \in B_{n_0}$.  We first verify that $\psi(x) \in A_{n_0}$.  Fix $n\geq n_0$ and $s \in \omega^{<\omega}$, with $p = \len s$.  If $s \neq \alpha_{n+p} \upto p$, then $\psi (x) (n,s) = (\mathbf 1_0 , \ldots , \mathbf 1_{p-1}) \in K_0 \upto p$, since $\bar{\bf 1} \in K_0$.  On the other hand, if $s = \alpha_{n+p} \upto p$, then 
\[
\psi (x) (n,s) = \psi_n (x) \upto p \in K_0 \upto p,
\]
since $\psi_n (B_{n_0}) \subseteq \psi_n (B_n) \subseteq K_0$, by assumption.  As $n\geq n_0$ and $s$ were arbitrary, we see that $\psi (x) \in A_{n_0}$.

It remains to see that $\psi (x) \in \langle \tilde A\rangle$.  Naturally, it suffices to prove that $\psi(x) \in \tilde A$.  We must show that, for all but finitely many $n,s$, 
\begin{equation}\label{E5}
(\psi (x) (n,s))^\ast \leq s
\end{equation}
Fix $n,s$ with $p = \len s$.  If $s \neq \alpha_{n+p} \upto p$, then \eqref{E5} follows, since $\psi (x) (n,s)  = (\mathbf 1_0 , \ldots , \mathbf 1_{p-1})$ and $(\mathbf 1_0 , \ldots , \mathbf 1_{p-1})^\ast = 0^p$.  If $s= \alpha_{n+p} \upto p$ and $n +p \geq n_0$, then $(\psi_n (x))^\ast \leq \alpha_{n+p}$, since $x \in B_{n_0} \subseteq B_n$ and $n \leq n+p$.  Hence
\begin{align*}
(\psi (x) (n,s))^\ast
&= \psi_n(x)^\ast \upto p\\
&\leq \alpha_{n+p} \upto p\\
&= s
\end{align*}
and \eqref{E5} holds for $n,s$.  We see that \eqref{E5} only fails when $n+\len s <n_0$ and $s = \alpha_{n+\len s} \upto \len s$.  There are only finitely many such pairs $n,s$.  

We have shown that $\psi (x) \in \tilde A$ and hence $\psi (x) \in \tilde A \cap A_{n_0} \subseteq H_0$.  This completes the proof.
\end{proof}


\section{Universal $F_\sigma$ subgroups for $K_\sigma$}\label{S5}

Theorem~\ref{T1} gives a universal $K_\sigma$ subgroup of $G^\omega$ whenever $G^\omega$ is locally compact.  If $G$ is arbitrary, we can still prove that there is an $F_\sigma$ subgroup of $G^\omega$ which is universal for $K_\sigma$ subgroups of $G^\omega$ (Theorem~\ref{T2} from Section~\ref{S1}).  In Section~\ref{S6}, we will show that $S_\infty^\omega$ does not have a universal $K_\sigma$ subgroup, implying that this result cannot, in general, be improved.

We recall Theorem~\ref{T2} and give its proof below.

\begin{refthmT2}
For every Polish group $G$, there is an $F_\sigma$ subgroup $H \subseteq G^\omega$ which is universal for $K_\sigma$ subgroups of $G^\omega$.
\end{refthmT2}

Before proceeding, it is worth mentioning a corollary of Theorem~\ref{T2}.  Recall that a Polish group $\mathbb G$ is {\em universal} if every Polish group is isomorphic to a closed subgroup of $\mathbb G$.

\begin{corollary}\label{univ F subgroup}
If $\mathbb G$ is a universal Polish group, then there is an $F_\sigma$ subgroup $H_0 \subseteq \mathbb G$ such that, for any $K_\sigma$ subgroup $H$ of a Polish group $G$, there is a continuous injective group homomorphism $\varphi : G \rightarrow \mathbb G$ such that $H = \varphi^{-1} (H_0)$.
\end{corollary}

\begin{proof}
Let $\mathbb G$ be a universal Polish group and $\tilde H_0 \subseteq \mathbb G^\omega$ an $F_\sigma$ subgroup which universal for $K_\sigma$ subgroups of $\mathbb G^\omega$.  By the universality of $\mathbb G$, we may identify $\tilde H_0$ with an $F_\sigma$ subgroup $H_0 \subseteq \mathbb G$.  Observe that, since $\mathbb G$ itself is isomorphic to a closed subgroup of $\mathbb G^\omega$, $H_0$ is universal for $K_\sigma$ subgroups of $\mathbb G$.  

Fix any Polish group $G$ and $H \subseteq G$, a $K_\sigma$ subgroup.  Let $\pi : G \rightarrow \mathbb G$ be an isomorphic embedding.  Note that $\pi (H)$ is a $K_\sigma$ subgroup of $\mathbb G$ and hence there is a continuous homomorpism $\varphi : \mathbb G \rightarrow \mathbb G$ such that $\varphi^{-1} (H_0) = \pi (H)$.  Inspecting the proof of Theorem~\ref{T2} below, it will become apparent that $\varphi$ can be chosen to be injective.  Since $\pi$ is injective, it follows that $(\varphi \circ \pi)^{-1} (H_0) = H$.
\end{proof}

It is a theorem of V. V. Uspenski\u\i ~(see Kechris \cite[9.18]{KECHRIS.dst}) that there are universal Polish groups.  In particular, the homeomorphism group of the Hilbert cube is a universal Polish group.

\medskip

\noindent{\it Notation.}  For the sake of the next proof, we introduce some notation.  If $w$ is a group word, let $w^{-1}$ denote the shortest group word such that $(w(a_0 , \ldots , a_m))^{-1} = w^{-1} (a_0 , \ldots , a_m)$, for any group elements $a_0 , \ldots , a_m$.  (For example, if $w (a,b) = ab$, then $w^{-1} (a,b) = b^{-1} a^{-1}$.)  If $w_1 , w_2$ are group words, let $w_1 w_2$ denote the concatenation of $w_1$ and $w_2$, i.e., for any group elements $a_0 , \ldots , a_m$ and $b_0 , \ldots , b_n$, 
\[
w_1 w_2 (a_0 , \ldots , a_m , b_0 , \ldots , b_n) = w_1 (a_0 , \ldots , a_m) \cdot w_2 (b_0 , \ldots , b_n).
\]

\begin{proof}[Proof of Theorem~\ref{T2}]
Fix a Polish group $G$, with compatible metric $d$ and identity element $\mathbf 1$.  For $x \in G$ and $A \subseteq G$, let $\dist (x , A) = \inf \{ d (x,y) : y \in A\}$.  Let $\mathcal B$ be a countable topological basis for $G$ and let $\mathcal F$ denote the set of all finite families $F \subseteq \mathcal B$ such that $\mathbf 1 \in \bigcup F$.

Since our objective is to produce a subgroup of $G^\omega$ with the desired universality property, we will simplify notation by assuming that $G$ is itself a countable power (i.e., $G\cong G^\omega$) and show that there is an $F_\sigma$ subgroup of $G^\omega$ of which every $K_\sigma$ subgroup of $G$ is a continuous homomorphic pre-image.  Specifically, we will work with an isomorphic copy of $G^\omega$ in the form of $G^{\mathcal F \times \omega}$.

Define an $F_\sigma$ subset $H$ of $G^{\mathcal F \times \omega}$ by letting $\xi \in H$ if, and only if, there exist an $n \in \omega$ and an $(m+1)$-ary group word $w$ such that, for each $F \in \mathcal F$ and $n' \geq n$, 
\[
\xi (F,n') \in \cl ( \{ w(x_0 , \ldots , x_m) : x_0 , \ldots , x_m \in \bigcup F \} ) .
\]
The next two claims will complete the proof.

\begin{claim1}
$H$ is a subgroup of $G^{\mathcal F \times \omega}$.
\end{claim1}

\proofclaim  Let $\bar{\mathbf 1}$ denote the identity element of $G^{\mathcal F \times \omega}$.  We have $\bar{\mathbf 1} \in H$, witnessed by $n=0$ and the word $w(a) = a$, since $\mathbf 1 \in \bigcup F$, for each $F \in \mathcal F$.  

Closure under taking inverses follows from from the observation that if $n$ and $w$ witness $\xi \in H$, then $n$ and $w^{-1}$ will witness $\xi^{-1} \in H$.

Suppose that $\xi_1 , \xi_2 \in H$.  Let $n_1 , n_2 \in \omega$ and $w_1 , w_2$ be group words as in the definition of $H$, witnessing the membership of $\xi_1$ and $\xi_2$ in $H$, respectively.  Let $n = \max \{ n_1 , n_2\}$ and $w = w_1 w_2$.  It follows from the definition of $H$ that the pair $n,w$ witnesses the membership of the product $\xi_1 \xi_2$ in $H$.

\begin{claim2}
Every $K_\sigma$ subgroup of $G$ is a continuous homomorphic pre-image of $H$.
\end{claim2}

\proofclaim  Fix a $K_\sigma$ subgroup $\bigcup_n K_n$ of $G$, with each $K_n$ compact and $\mathbf 1 \in K_0 \subseteq K_1 \subseteq \ldots$.  For each $n,k$, let $F_{n,k} \in \mathcal F$ be a finite $\frac{1}{k}$-cover of $K_n$ (i.e., each $U \in F_{n,k}$ has diameter less than $1/k$ and $K_n \subseteq \bigcup F_{n,k}$) with the property that $U \cap K_n \neq \emptyset$, for each $U \in F_{n,k}$.  Define a continuous homomorphism $\psi : G \rightarrow G^{\mathcal F \times \omega}$ by 
\[
\psi (x) (F,n) = \begin{cases}
x &\mbox{if } (\exists k) (F = F_{n,k}),\\
\mathbf 1 &\mbox{otherwise}.
\end{cases}
\]

We wish to see that $\bigcup_n K_n = \psi^{-1} (H)$.  Suppose first that $x \in K_n$ and hence $x \in K_{n'}$, for each $n' \geq n$.  If $w$ is the group word $w(a) = a$, then $n$ and $w$ witness $\psi (x) \in H$.  Indeed, if $n' \geq n$ and, for some $k$, we have $F = F_{n,k}$, then $\psi (x) (F,n') = x \in K_{n'} \subseteq \bigcup F_{n',k}$.  On the other hand, if, for every $k$, we have $F \neq F_{n',k}$, then $\psi (x) (F,n) = \mathbf 1 \in \bigcup F$, since $F \in \mathcal F$.

Suppose now that $\psi (x) \in H$, witnessed by $n,w$, where $w$ is $(m+1)$-ary.  This implies that, for each $k \in \omega$, 
\[
x = \psi (x) (F_{n,k},n) \in \cl (\{ w (x_0 , \ldots , x_m) : x_0 , \ldots , x_m \in \bigcup F_{n,k}\}).
\]
Hence, by the choice of $F_{n,k}$, there exist $x^k_0 , \ldots , x^k_m \in G$ such that, for each $k \in \omega$, 
\begin{enumerate}
\item $d(x , w(x^k_0 , \ldots , x^k_m)) \leq 1/k$ and
\item $\dist (x^k_i , K_n) \leq 1/k$, for each $i \leq m$.
\end{enumerate}
Since $K_n$ is compact, condition 2 implies that there are elements $y_0 , \ldots , y_m \in K_n$ and a subsequence $k_0 < k_1 < \ldots$ such that, for each $i \leq m$, 
\[
x^{k_p}_i \rightarrow y_i \mbox{ as } p \rightarrow \infty.
\]
More explicitly, for each pair $k,i$, there is an element $z^k_i \in K_n$ with 
\[
d(x^k_i , z^k_i) \leq 1/k.
\]
Hence one may iteratively (for each $i \leq m$) take convergent subsequences of $(z^k_i)_{k \in \omega}$ to obtain a common subsequence $k_0 < k_1 < \ldots$ such that $(z^{k_p}_i)_{p \in \omega}$ converges, for each $i \leq m$.  Letting $y_i = \lim_p z^{k_p}_i$, we have $y_i \in K_n$ and $y_i = \lim_p x^{k_p}_i$.  

It remains only to show that $x = w(y_0 , \ldots , y_m)$ and thereby conclude that $x \in \langle K_n \rangle \subseteq H$.  To this end, fix $\varepsilon > 0$.  Let $p_0$ be such that if $p \geq p_0$, we have
\begin{enumerate}
\item $1/k_p < \varepsilon /2$ and
\item $d (w (x^{k_p}_0 , \ldots , x^{k_p}_m) , w (y_0 ,\ldots , y_m)) < \varepsilon /2$.
\end{enumerate}
Such a $p_0$ exists satisfying the second condition since $w$ determines a continuous map $G^m \rightarrow G$ and each sequence $(x^{k_p}_i)_{p \in \omega}$ converges to $y_i$.  For each $p \geq p_0$, we now have 
\begin{align*}
d ( x , w (y_0 ,\ldots , y_m))
&\leq d ( x , w (x^{k_p}_0 , \ldots , x^{k_p}_m) ) + d ( w (x^{k_p}_0 , \ldots , x^{k_p}_m) , w (y_0 , \ldots , y_m) )\\
&< 1/k_p + \varepsilon / 2\\
&< \varepsilon.
\end{align*}
Since $\varepsilon$ was arbitrary, we have shown that $x = w (y_0 , \ldots , y_m)$.  This completes the claim and concludes the proof.
\end{proof}



\section{The example of $S_\infty$}\label{S6}

In this section we prove that there is a countable power with no universal $K_\sigma$ subgroup.  In particular, we show that $S_\infty^\omega$ has no universal $K_\sigma$ subgroup (Theorem~\ref{T3}).  This suggests that Theorem~\ref{T1} cannot readily be expanded to a larger class of Polish groups.  In some sense, the example of $S_\infty$ also serves as a complement to Theorem~\ref{T2}, again suggesting that this may be a ``best possible'' result.

We state the main result of this section.

\begin{theorem}\label{non loc cpt example}
Their is no $K_\sigma$ subgroup of $S_\infty^\omega$ which is universal for compactly generated subgroups of $S_\infty^\omega$.
\end{theorem}

This theorem shows (in a strong way) that $S_\infty^\omega$ has neither universal compactly generated nor universal $K_\sigma$ subgroups.  Since $S_\infty^\omega$ embeds in $S_\infty$ as a closed subgroup and {\it vice versa}, it follows from Proposition~\ref{P2} that it will suffice to prove Theorem~\ref{non loc cpt example} for $S_\infty$, in place of $S_\infty^\omega$.

Recall that the topology on $S_\infty$ is generated by the basic clopen sets 
\[
\mathcal U (s) = \{ s \in S_\infty : s \subset f\},
\]
where $s : \omega \rightarrow \omega$ is a finite injection.  Hence the sets $\mathcal U ({\rm id} \upto n)$ form a neighborhood basis at the identity.  Becuase we will refer to these open sets several times in what follows, we write $\mathcal U_n$ for $\mathcal U ({\rm id} \upto n)$.

The fundamental elements of $S_\infty$ are cycles.  We use the notation $[a_1 , \ldots , a_n]$ for $n$-cycles and $[\ldots , a_{-1} , a_0 , a_1 , \ldots]$ for $\infty$-cycles.  For $\pi \in S_\infty$, we let 
\[
\supp (\pi) = \{ n : \pi (n) \neq n\} = \{ n : \pi^{-1}(n) \neq n\}.
\]

For any $f \in \omega^\omega$ (viewed as a function $\omega \rightarrow \omega$) we write $f^p$ for the $p$-fold composite of $f$ with itself, e.g., $f^2 = f \circ f$.  We will use this notation both for permutations of $\omega$ as well as arbitrary functions on $\omega$.

For each $\alpha \in \omega^\omega$, define 
\[
K_\alpha = \{ f \in S_\infty : f,f^{-1} \leq \alpha\},
\]
where ``$f\leq \alpha$'' signifies that, for each $n \in \omega$, $f(n) \leq \alpha(n)$.  Note that each $K_\alpha$ is compact in $S_\infty$ and that every compact subset of $S_\infty$ is contained in some $K_\alpha$.  Suppose that $H = \bigcup_n K_n$ is a $K_\sigma$ subset of $S_\infty$, with each $K_n$ compact.  To show that $H$ is not universal for compactly generated subgroups of $S_\infty$, it will suffice to find a compact set $K \subseteq S_\infty$ such that no homomorphic image of $K$ is contained in $H$.  For this, it is enough to assume that each $K_n$ has the form $K_{\beta_n}$, for some $\beta_n \in \omega^\omega$.  We therefore show

\begin{theorem}\label{no univ K sigma}
Given $\beta_0 , \beta_1 , \ldots \in \omega^\omega$, there exists $\alpha \in \omega^\omega$ such that, for each continuous injective group homomorphism $\Phi : S_\infty \rightarrow S_\infty$, we have $\Phi (K_\alpha) \nsubseteq \bigcup_n K_{\beta_n}$.
\end{theorem} 

We require a few lemmas.

\begin{lemma}\label{injecto}
If $\Phi$ is a continuous endomorphism of $S_\infty$ and $\ker (\Phi) \neq S_\infty$, then $\Phi$ is injective.
\end{lemma}

\begin{proof}
Since $\Phi$ is continuous, $\ker(\Phi)$ is a closed normal subgroup of $S_\infty$.  On the other hand, it is a theorem of Schreier-Ulam that the only normal subgroups of $S_\infty$ are $\{ \id \}$, the infinite alternating group, the group of finite support permutations and $S_\infty$ itself.  Of these, only $\{ \id \}$ and $S_\infty$ are closed.
\end{proof}

Noting that every $K_\sigma$ subgroup of $S_\infty$ is a proper subgroup, Lemma~\ref{injecto} implies that a group-homomorphism reduction between $K_\sigma$ subgroups of $S_\infty$ must be injective.  (This follows from the fact that, if $A = \varphi^{-1} (B)$, then $\ker (\varphi) \subseteq A$.)

\begin{lemma}\label{fin ext}
Suppose that $\alpha \in \omega^\omega$ is such that $(\forall n) (\alpha(n) \geq n)$. If $f \in K_\alpha$ and $s \subset f$ is a finite injection, then there is a finite support permutation $\pi \in K_\alpha$ such that $s \subset \pi$.
\end{lemma}

\begin{proof}
Let $\mathcal S$ be the set of cycles $\sigma \subset f$ such that $\supp (\sigma)$ intersects the domain or range of $s$.  Since $s$ is a finite function, $\mathcal S$ is a finite set of disjoint cycles.  For each $\infty$-cycle $\tau \in \mathcal S$, we will define a finite cycle $\tau^\ast \in K_\alpha$ such that $\tau^\ast$ agrees with $\tau$ on $\dom (s) \cup \ran (s)$.  Write $\tau$ as 
\[
[\ldots a_{-1} , a_0 , a_1 , \ldots ].
\]

Let $n_0 , n_1 \in \mathbb Z$ be such that $n_0 \leq n_1$ and if $a_i \in {\rm dom} (s) \cup {\rm ran} (s)$, for some $i$, then $n_0 \leq i \leq n_1$.  By taking $n_1$ large enough, we may assume that $a_{n_1} > a_{n_0}$.  Let $m \leq n_0$ be large enough that $a_m < a_{n_1}$ and $a_{m-1} > a_{n_1}$.  (Note that we have strict inequalities since $\tau$ is an $\infty$-cycle and hence all $a_i$ are distinct.)  Define 
\[
\tau^\ast = [a_m , \ldots , a_{n_0} , \ldots , a_{n_1}].
\]

We will verify that $\tau^\ast \in K_\alpha$, i.e., $\tau^\ast , (\tau^\ast)^{-1}$ are both bounded by $\alpha$.  Since $\tau \subset f \in K_\alpha$, we know that $\tau, \tau^{-1}$ are bounded by $\alpha$.  Hence we need only check that $a_{n_1} \leq \alpha (a_m)$ and $a_m \leq \alpha(a_{n_1})$, since $\tau^\ast$ agrees with $\tau$, except at $a_{n_1}$.  That $a_{n_1} \leq \alpha (a_m)$ follows from 
\[
a_{n_1} < a_{m-1} = \tau^{-1} (a_m) \leq \alpha(a_m).
\]
(We are using the fact that $\tau^{-1} \leq \alpha$.)  On the other hand, $a_m \leq \alpha(a_{n_1})$ follows from the fact that 
\[
a_m < a_{n_1} \leq \alpha (a_{n_1}),
\]
by our assumption that $(\forall n) (\alpha (n) \geq n)$.

We may now define the desired $\pi$ as in the statement of the lemma.  Let $\pi$ be the product of all finite cycles in $\mathcal S$ together with all $\tau^\ast$, for $\infty$-cycles $\tau \in \mathcal S$.
\end{proof}

\begin{lemma}\label{single beta}
If $\Phi$ is a continuous endomorphism of $S_\infty$ and $\alpha , \beta_m \in \omega^\omega$ are such that $(\forall n) (\alpha (n) \geq n)$, $\Phi (K_\alpha) \subseteq \bigcup_n K_{\beta_m}$ and $\{ \beta_m : m \in \omega\}$ is closed under composition, then there exist $n,m \in \omega$ such that 
\[
\Phi (\mathcal U_n \cap K_\alpha) \subseteq K_{\beta_m}.
\]
\end{lemma}

\begin{proof}
The following claim is the core of the proof.

\begin{claim}
There exists a finite support permutation $\pi \in K_\alpha$ and $m \in \omega$ such that $\Phi (\mathcal U(\pi) \cap K_\alpha) \subseteq K_{\beta_m}$.
\end{claim}

\proofclaim  Let $C$ be the compact set $\Phi (K_\alpha)$.  Applying the Baire Category Theorem to $C$, it follows that exists $m \in \omega$ such that $K_{\beta_m} \cap C$ is non-meager relative to $C$.  As $K_{\beta_m}$ is closed, this implies that there exists a nonempty open set $\mathcal V \subseteq S_\infty$ such that $\mathcal V \cap C \subseteq K_{\beta_m}$.  Let $\mathcal U = \Phi^{-1} (\mathcal V)$.  Since $\mathcal U \cap K_\alpha \neq \emptyset$, there is a finite injection $s : \omega \rightarrow \omega$ with $\mathcal U(s) \subseteq \mathcal U$ and $\mathcal U(s) \cap K_\alpha \neq \emptyset$.  Lemma~\ref{fin ext} thus yields a finite support permutation $\pi \in K_\alpha$ such that $\mathcal U (\pi) \subseteq \mathcal U$.  Hence 
\[
\Phi (\mathcal U(\pi) \cap K_\alpha) \subseteq \Phi (\mathcal U \cap K_\alpha) \subseteq (\mathcal V \cap C) \subseteq K_{\beta_m}.
\]
This completes the claim.

\medskip

Suppose that $\pi , m$ are as in the claim, such that $\Phi (\mathcal U(\pi) \cap K_\alpha) \subseteq K_{\beta_m}$.  If $n$ is an upper bound for the support of $\pi$, then $\pi^{n!} = \id$.  Note that each permutation in $\mathcal U(\pi) \cap K_\alpha$ has the form $\pi \circ f$, for some $f \in K_\alpha$, with $\supp (f)$ disjoint from $\supp (\pi)$.  With this in mind, fix an arbitrary $\pi \circ f \in \mathcal U (\pi) \cap K_\alpha$ and observe that 
\[
\pi^{n! -1} \circ \pi \circ f = f
\]
and hence $\Phi (f)$ is the composite of $n!$ elements of $K_{\beta_m}$, since $\Phi (\pi) , \Phi (\pi \circ f) \in K_{\beta_m}$.  On the other hand, any composite of $n!$ elements of $K_{\beta_m}$ is bounded by the $n!$-fold composite of $\beta_m$ with itself.  As we assumed that $\{ \beta_m : m \in \omega\}$ is closed under composition, we conclude that $\Phi (\mathcal U_n \cap K_\alpha) \subseteq K_{\beta_{r}}$, for an appropriate $r \in \omega$.  This completes the proof.
\end{proof}

\begin{lemma}\label{big alpha}
Given increasing $\beta_m \in \omega^\omega$, there exists $\alpha \in \omega^\omega$ such that, for each $n,m \in \omega$, there is no continuous injective group homomorphism $\Phi$ of $S_\infty$ with $\Phi (\mathcal U_n \cap K_\alpha) \subseteq K_{\beta_m}$.
\end{lemma}

\begin{proof}
For the sake of the present proof, if $f \in S_\infty$, we define a {\em chain of roots of length $n$} for $f$ to be a sequence $f_0 , \ldots , f_n \in S_\infty$ such that $f_0 = f$ and $f_j^2 = f_{j-1}$, for each $1 \leq j \leq n$.

Suppose that $f \in K_\alpha$ is a product of disjoint $2^k$-cycles, for some $k \geq 1$.  If $n \in \supp (f)$ (i.e., $f(n) \neq n$), then $f$ has no chain of roots in $K_\alpha$, of length greater than $\alpha (n)$.  This follows from the fact that, were $f_0 , \ldots , f_p$ a chain of roots of length $p > \alpha (n)$, then at least one $f_j$ is not a member of $K_\alpha$, as $f_0 (n) , \ldots , f_p(n)$ are all distinct.  Recall here that ``square-roots'' of products of disjoint $2^k$-cycles are obtained by interleaving terms of distinct cycles to form permutations containing products of disjoint $2^{k+1}$-cycles.  (This is a consequence of the fact that, if $\sigma$ is a $2^m$-cycle, for some $m$, then $\sigma^2$ is a product of two disjoint $2^{m-1}$-cycles.)  

As an example,
\begin{align*}
f_0 &= [0,1] \, [2,3] \, [4,5] \, [6,7]\\
f_1 &= [0,2,1,3] \, [4,6,5,7]\\
f_2 &= [0,4,2,6,1,5,3,7]
\end{align*}
is a chain of roots for $f_0$, of length 2.  We remark that this behavior does not apply to cycles of other lengths.  For instance, in the case of 3-cycles, one has $[1,3,2]^2 = [1,2,3]$.

Let $\alpha \in \omega^\omega$ be such that, for each $k \in \omega$, the permutation 
\[
f_k = [k,k+1] \, [k+2,k+3] \ldots
\]
has a chain of roots in $K_\alpha$ of length at least 
\[
\max_{i \leq k} (\beta_i^{4k+1}(k)) + 1.
\]
We may further assume that $(\forall k) (\alpha (k) \geq k+2)$.

Suppose, towards a contradiction, that $\Phi$ is a continuous endomorphism of $S_\infty$ with $\Phi (\mathcal U_n \cap K_\alpha) \subseteq K_{\beta_m}$, for some $\alpha \in \omega^\omega$ and $m,n \in \omega$.  For simplicity, write $\beta_m = \beta$.  Let $a \in \omega$ be least such that $\Phi ([n,n+1,n+2]) (a) \neq a$.

For each $k \geq n+2$, we have that 
\[
\Phi ([n,n+1,k]) = \Phi ([n+2,k]) \circ \Phi ([n,n+1,n+2]) \circ \Phi ([n+2,k]).
\]
Observe that 
\[
[n+2,k] = [n+2,n+3] \, [n+3 , n+4] \ldots [k-1,k] \, [k-2,k-1] \ldots [n+2,n+3]
\]
and hence $[n+2,k]$ is a product of fewer than $2k$ members of $K_\alpha$, since each $[j,j+1]$ is in $K_\alpha$.  Thus $\Phi ([n+2,k])$ is a product of fewer than 2k members of $K_\beta$.  (Since each $\Phi([j,j+1]) \in K_\beta$, for each $j \geq n$.)  In particular, $\Phi ([n+2,k])$ is bounded by $\beta^{2k}$, the $2k$-fold composite of $\beta$ with itself.  (This follows in part from the fact that $\beta$ was assumed to be increasing.)  Hence we have
\[
\Phi ([n+2,k])^{-1} (a) = \Phi ([n+2,k]) (a) \leq \beta^{2k} (a)
\]
and thus there exists $b_0 \in \supp (\Phi([n,n+1,k]))$ with $b_0 \leq \beta^{2k} (a)$. 

As noted above, the choice of $\alpha$ guarantees that each $[j,j+1] \in K_\alpha$.  Hence
\[
[k,k+1] \, [k+2,k+3] \ldots \in K_\alpha
\]
and thus 
\[
h = \Phi ([k,k+1] \, [k+2,k+3] \ldots) \in K_\beta.
\]
Observe now that 
\begin{align*}
&\Phi ([n,n+1,k,k+1] \, [k+2,k+3] \, [k+4,k+5] \ldots) \qquad (\ast)\\
&= \Phi ([n,n+1,k] \, [k,k+1] \, [k+2,k+3] \ldots)\\
&= \Phi ([n,n+1,k]) \circ \Phi ([k,k+1] \, [k+2,k+3] \ldots)\\
&= \Phi ([n,n+1,k]) \circ  h
\end{align*}
As can be seen from the line marked $(\ast)$, this permutation has order 4, while $\Phi ([n,n+1,k])$ has order 3.  Thus $\supp (h)$ must intersect each orbit of $\Phi ([n,n+1,k])$, as otherwise the permutation above will contain a 3-cycle and not be of order 4.  In particular, $\supp (h)$ contains an element of the orbit of $b_0$ under $\Phi ([n,n+1,k])$.  This implies that $\supp (h)$ contains an element $b_1 \leq \beta^{4k} (a)$.  As $h$ is of order 2, it must be a product of disjoint $2$-cycles.  We now conclude that $h$ has no chain of roots in $K_\beta$, of length greater than $\beta (b_1) \leq \beta^{4k+1} (a)$.  (Again, we are using the fact that $\beta$ is increasing to obtain this inequality.)

On the other hand, if $k \geq m,a$, then
\[
\beta^{4k+1} (a) \leq \beta^{4k+1}(k) = \beta_m^{4k+1}(k) \leq \max_{i \leq k} (\beta_i^{4k+1}(k))
\]
and $[k,k+1] \, [k+2,k+3] \ldots$ has a chain of roots in $K_\alpha$ of length at least 
\[
\max_{i \leq k} (\beta_i^{4k+1}(k)) +1.
\]
This is a contradiction since $\Phi$ maps $K_\alpha$ into $K_\beta$ and, being a homomorphism, must preserve chains of roots.
\end{proof}

We may now complete the proof of Theorem~\ref{no univ K sigma}.

\begin{proof}[Proof of Theorem~\ref{no univ K sigma}]
Suppose that $\beta_0 , \beta_1 , \ldots \in \omega^\omega$ are given.  With no loss of generality, we may assume that $\{ \beta_m : m \in \omega\}$ is closed under compositions and that each $\beta_m$ is strictly increasing.  (Making these assumptions only enlarges the $K_\sigma$ set $\bigcup_m K_{\beta_m}$.  Also, these two assumptions do not conflict as the composite of increasing functions remains increasing.)  

Let $\alpha \in \omega$ be as in Lemma~\ref{big alpha}, for $\{ \beta_m : m \in \omega\}$.  Here we may assume that $\alpha(n) \geq n$, for each $n$.  If there is a continuous endomorphism $\Phi$ of $S_\infty$ such that $\Phi (K_\alpha) \subseteq \bigcup_m K_{\beta_m}$, then Lemma~\ref{single beta} yields $m,n$ such that $\Phi (\mathcal U_n \cap K_\alpha) \subseteq K_{\beta_m}$.  This contradicts the properties of $\alpha$.
\end{proof}  



\section{Examples}\label{S7}

\subsection{Basic examples}

We restate a proposition from the Introduction which will be our main tool in this section.

\begin{refpropP1}
Suppose that $G_1$ and $G_2$ are topological groups such that there exist continuous injective homomorphisms $\varphi_1 : G_1 \rightarrow G_2$ and $\varphi_2 : G_2 \rightarrow G_1$.  Let $\mathcal C$ be a class of subgroups which is closed under continuous homomorphic images.  If $G_1$ has a universal $\mathcal C$ subgroup, then $G_2$ also has a universal $\mathcal C$ subgroup.
\end{refpropP1}

The following examples are direct applications of Proposition~\ref{P1}, together with Theorem~\ref{T1}.

\begin{example} Let $\mathbf c_0 \subset \mathbb R^\omega$ be the subgroup
\[
\{ x \in \mathbb R^\omega : \lim_n x(n) = 0\}.
\]
Recall that $\mathbf c_0$ is a separable Banach space (hence a Polish group) when equipped with the sup-norm (denoted by $\lVert \cdot \rVert_{\rm sup}$).  Let $\mathcal C$ be either the class of compactly generated or $K_\sigma$ subgroups.  Since $\mathbf c_0$ is nowhere locally compact, Theorem~\ref{T1} does not immediately give universal a $\mathcal C$ subgroup of $\mathbf c_0^\omega$.  Nonetheless, we shall see that $\mathbf c_0^\omega$ has a universal $\mathcal C$ subgroup.

The Banach space topology on $\mathbf c_0$ refines the subspace topology inherited from $\mathbb R^\omega$.  To see this, suppose that $U = I_0 \times \ldots \times I_{k-1} \times \mathbb R^\omega$ is a basic open set in $\mathbb R^\omega$ (where $I_0, \ldots , I_{k-1} \subseteq \mathbb R$ are bounded open intervals) and $x_0 \in U \cap \mathbf c_0$.  Let $\varepsilon >0$ be small enough that, for each $n<k$, the open interval $(x_0(n)-\varepsilon , x_0(n)+\varepsilon)$ is contained in $I_n$.  If 
\[
B = \{ x \in \mathbf c_0 : \lVert x - x_0 \rVert_{\rm sup} < \varepsilon\},
\]
then $B$ is open in $\mathbf c_0$ and $x \in B \subseteq U \cap \mathbf c_0$.  Hence $U \cap \mathbf c_0$ is open with respect to the Banach space topology on $\mathbf c_0$.  This implies that the inclusion map $\mathbf c_0 \rightarrow \mathbb R^\omega$ is a continuous injective homomorphism and hence so is the inclusion $\mathbf c_0^\omega \rightarrow \mathbb R^{\omega \times \omega} \cong \mathbb R^\omega$.  

To apply Proposition~\ref{P1}, we also need to check that there is a continuous injective homomorphism of $\mathbb R^\omega$ into $\mathbf c_0^\omega$.  Indeed, this is witnessed by the map $\varphi : \mathbb R^\omega \rightarrow \mathbf c_0^\omega$ where 
\[ 
\varphi(x)(n) = (x(n),0,0, \ldots).
\]
By Proposition~\ref{P1} we conclude that $\mathbf c_0^\omega$ has a universal $\mathcal C$ subgroup, since $\mathbb R^\omega$ does.
\end{example}

\begin{remark*}
In the previous example, we do not claim that $\RR^\omega$ is isomorphic to a subgroup of $\mathbf c_0$, nor {\it vice versa}, as these would be false statements.
\end{remark*}

By similar arguments using the fact that the Banach space topologies of $\ell^p , \ell^\infty ,\mathbf c \subset \mathbb R^\omega$ refine their subspace topologies, we can also conclude that the groups $(\ell^p)^\omega$, $(\ell^\infty)^\omega$ and $\mathbf c^\omega$ contain universal subgroups for the classes of $K_\sigma$ and compactly generated subgroups.  The case of $(\ell^\infty)^\omega$ is interesting because $\ell^\infty$ (with the sup-norm) is complete, but not separable, i.e., not a Polish space.\footnote{Definitions of the Banach spaces $\ell^p$, $\ell^\infty$ and $\mathbf c$ may be found in Conway \cite{CONWAY.functional}.}

It is also worth mentioning the case of $\ell^2$.  Since $\ell^2$ is a separable Hilbert space and, by Corollary 5.5 in Conway \cite{CONWAY.functional}, all separable Hilbert spaces (over $\mathbb R$) are isomorphic, we have that all separable Hilbert spaces are isomorphic to $\ell^2$.  The comments above thus imply the following.

\begin{proposition}
The countable power of every separable Hilbert space (over $\mathbb R$) contains universal $K_\sigma$ and compactly generated subgroups.
\end{proposition}

\begin{remark*}
The arguments above apply equally to $\mathbb C$ in place of $\mathbb R$.  (I.e., $\mathbb C^\omega$ also has universal subgroups in these two classes.)  Thus the proposition above applies to complex Hilbert spaces as well.
\end{remark*}

The following example shows the existence of universal subgroups in another non-separable topological group.

\begin{example} Let $S$ be a separable space and $C(S)$ be the additive group of continuous real-valued functions on $S$, with the topology of uniform convergence.  The group $C(S)$ is metrizable, but not separable if $S$ is not compact.  A compatible metric is  
\[
\rho(f,g) = \sup \{\min \{|f(x) - g(x)| , 1\} : x \in S\}.
\]
The distance function $\rho$ is the so-called ``uniform metric'' on $C(S)$.\footnote{See Munkres \cite[p.~266]{MUN.topology}.}

Let $A \subseteq S$ be a countable dense set.  Consider the Polish group $\mathbb R^A$, equipped with the product topology, i.e., $\mathbb R^A \cong \mathbb R^\omega$.  The map $\psi : C(S) \rightarrow \mathbb R^A$ defined by 
\[
f \mapsto f \upto A
\]
is a group homomorphism.  To see that $\psi$ is continuous it suffices to check that $\psi^{-1}(U)$ is open when $U$ is a basic neighborhood of $\bar 0$.  Given a basic neighborhood $U \ni \bar 0$, we may assume that, for some finite set $F \subseteq A$ and $\varepsilon > 0$,
\[ 
U = \{ x \in \mathbb R^A :  (\forall a \in F) (|x(a)|<\varepsilon)\}.
\]
Let $\mathcal F = \{ f \in C(S) : (\forall a \in F) (f(a) = 0)\}$ and take 
\[
\mathcal V = \bigcup_{f \in \mathcal F} \{ g \in C(S) : \rho (f,g) < \varepsilon\}.
\]
As the union of open sets, $\mathcal V$ is open in $C(S)$ and $\psi^{-1} (U) = \mathcal V$.  Also, $\psi$ is injective because $A$ is dense and thus $f \upto A = g \upto A$ implies $f = g$.  It follows that $C(S)^\omega$ may be mapped into $\mathbb R^{A \times \omega} \cong \mathbb R^\omega$ as well, via a continuous injective group homomorphism.

Finally, note that $\mathbb R^\omega$ embeds in $C(S)^\omega$ (as a closed subgroup in this case) via the map $\varphi : \mathbb R^\omega \rightarrow C(S)^\omega$, where $\varphi(x)(n)$ is the constant function $f \equiv x(n)$.  Proposition~\ref{P1} now lets us conclude that $C(S)^\omega$ contains universal compactly generated and $K_\sigma$ subgroups. 
\end{example}

As noted in Kechris \cite[\S 12.E]{KECHRIS.dst}, every separable Banach space is isomorphic to a closed subspace of $C(2^\omega)$.  By the previous example, we therefore have

\begin{proposition}
Let $\mathcal C$ be either the classes of compactly generated or the class of $K_\sigma$ subgroups.  There is a subgroup $H_0 \subseteq C(2^\omega)^\omega$, with $H_0 \in \mathcal C$, such that for any separable Banach space $\mathfrak B$ and any subgroup $H \subseteq \mathfrak B^\omega$ in $\mathcal C$, there is a continuous group homomorphism $\varphi : \mathfrak B^\omega \rightarrow C(2^\omega)^\omega$ such that $H = \varphi^{-1} (H_0)$.
\end{proposition}

The next example relates directly to Theorem~\ref{T1}.

\begin{example}  Let $(G_n)_{n \in \omega}$ be a sequence of locally compact Polish groups.  Consider $\bigoplus_n G_n$ with the subspace topology from $\prod_n G_n$.  Although separable, the direct sum $\bigoplus_n G_n$ is, in general, not Polishable.\footnote{To see this with $G_n = \mathbb R^n$, suppose that $\mathcal T$ is a Polishing topology on $\bigoplus_n \mathbb R^n$.  By the Baire Category Theorem, there is an $n$ such that $\mathbb R^n$ is $\mathcal T$-non-meager in $\bigoplus_n \mathbb R^n$.  Being a subgroup, $\mathbb R^n$ is thus open in $\bigoplus_n \mathbb R^n$, by Pettis' theorem.  This gives a contradiction to separability, since $\mathbb R^n$ has uncountable index in $\bigoplus_n \mathbb R^n$.}

The product $\prod_n G_n^\omega$ is isomorphic to a closed subgroup of $(\bigoplus_n G_n)^\omega$.  Furthermore, $(\bigoplus_n G_n)^\omega$ is isomorphic the the $\mathbf \Pi^0_3$ subgroup
\[
\{ \xi  : (\forall k) (\forall^\infty n) (\xi (n) (k) = \mathbf 1_n) \}
\]
of $\prod_n G_n^\omega$.  Theorem~\ref{T1} together with Proposition~\ref{P1} therefore imply that $(\bigoplus_n G_n)^\omega$ has universal compactly generated and $K_\sigma$ subgroups.
\end{example}

\subsection{Separable Banach spaces}

In this section we show that every separable infinite-dimensional Banach space with an unconditional basis (we give the definition below) has universal compactly generated and $K_\sigma$ subgroups.  The key facts will be Proposition~\ref{P1} along with the following.

\begin{theorem}\label{c_0}
The Banach space $\mathbf c_0$ has universal compactly generated and universal $K_\sigma$ subgroups.
\end{theorem}

In each case, we obtain the desired universal subgroup of $\mathbf c_0$ by ``shrinking'' an appropriate universal subgroup of $\mathbb R^\omega$.  Note that we could also prove these facts directly by modifying the proof of Theorem~\ref{T1}.  We begin with a lemma.

\begin{lemma}\label{lemming}
Suppose $\alpha : \omega \rightarrow \mathbb R^+$ is such that $\lim_n \alpha (n) = 0$.  If $F \subseteq \mathbf c_0$ is closed and $|x(n)| \leq \alpha (n)$, for each $x \in F$ and $n \in \omega$, then $F$ is compact in $\mathbf c_0$.
\end{lemma}

\begin{proof}
Suppose that $(x_i)_{i \in \omega}$ is a sequence of elements of $F$.  Let $i_0, i_1 , \ldots$ be a subsequence such that $(x_{i_n} (k))_{n \in \omega}$ is convergent, for each $k \in \omega$.  Such a subsequence may be obtained by succesively choosing subsequences to guarantee that $(x_{i_n} (j))_{n \in \omega}$ is Cauchy for all $j \leq k$ and taking $(i_n)_{n \in \omega}$ to be a pseudo-intersection of these subsequences.  Let $x \in \mathbf c_0$ be given by $x(k) = \lim_n x_{i_n} (k)$, for each $k$.  Note that $|x(k)| \leq \alpha (k)$, for each $k \in \omega$.

To see that $\lVert x_{i_n} - x \rVert_{\rm sup} \rightarrow 0$, as $n \rightarrow \infty$, fix $\varepsilon > 0$ and let $k_0$ be large enough that $| \alpha (k) | < \frac{\varepsilon}{2}$, for each $k \geq k_0$.  Let $n_0$ be large enough that $|x_{i_n}(k) - x(k)| < \varepsilon$, for each $n \geq n_0$ and $k< k_0$.  It follows that $\lVert x_{i_n} - x\rVert_{\rm sup} < \varepsilon$, for each $n \geq n_0$.
\end{proof}

\begin{proof}[Proof of Theorem~\ref{c_0}]
We consider each of the statements in Theorem~\ref{c_0} separately.

{\bf Compactly generated subgroups.}  Let $\langle K \rangle \subseteq \mathbb R^\omega$ be a universal compactly generated subgroup of $\mathbb R^\omega$.  (Such a subgroup exists by Theorem~\ref{T1}(1).)  With no loss of generality, we assume that the compact set $K$ contains $\bar 0$.  Let $\{ I_{n,p} : n,p \in \omega\}$ be intervals partitioning $\omega$ such that each $I_{n,p}$ has length $n$.  Define $K' \subseteq \mathbb R^\omega$ by 
\[
x \in K' \iff (\forall n,p) (x \upto I_{n,p} \in (1/np) K \upto n).
\]
Where $(1/np) K \upto n$ denotes the set of scalar multiples by $(1/np)$ of elements of $K \upto n$.  It follows from Lemma~\ref{lemming} that $K'$ is compact in $\mathbf c_0$.

We will show that $\langle K' \cap \mathbf c_0 \rangle$ is a universal compactly generated subgroup of $\mathbf c_0$.  Indeed, fix an arbitrary compact $A \subseteq \mathbf c_0$.  Since $A$ is also compact in $\mathbb R^\omega$, there is a continuous group homomorphism $\varphi : \mathbb R^\omega \rightarrow \mathbb R^\omega$ such that $\langle A \rangle = \varphi^{-1} (\langle K \rangle )$.\footnote{As noted earlier the Banach space topology of $\mathbf c_0$ refines the subspace topology inherited from $\mathbb R^\omega$ and hence compactness is ``preserved upwards.''}

For each $n \in \omega$, let $\tau (n) \in \omega\setminus \{0\}$ be such that, for every $x \in [-1,1]^\omega$ and $i < n$, we have $|\varphi (x) (i)| \leq \tau (n)$.  (Such $\tau(n)$ exist by the compactness of $[-1,1]^\omega$ and the continuity of $\varphi$.)  Define $\psi : \mathbb R^\omega \rightarrow \mathbb R^\omega$ by 
\[
\psi(x) \upto I_{n,p} = \begin{cases}
(1/np) \varphi (x) \upto n \quad &\mbox{if } p = \tau(n)^2\\
0^n &\mbox{otherwise}
\end{cases}
\]

\begin{claim1}
$\psi (\mathbf c_0) \subseteq \mathbf c_0$.
\end{claim1}

\proofclaim  Note that all continuous group homomorphisms of $\mathbb R^\omega$ are automatically linear, hence $\psi$ is linear.  Thus, to prove the claim, it will suffice to show that $\psi (x) \in \mathbf c_0$, for all $x \in \mathbf c_0$ with $\lVert x \rVert_{\rm sup} \leq 1$.  Fix such an $x$ and an $\varepsilon > 0$.  For $i \in \omega$, $\psi (x) (i) \neq 0$ only if $i \in I_{n , \tau(n)^2}$, for some $n$.  For $i \in I_{n ,\tau(n)^2}$, we have 
\begin{align*}
|\psi (x) (i)|
&\leq (1/n\tau(n)^2) \max_{j<n} |\varphi(x)(j)|\\
&\leq 1/n\tau(n)
\end{align*}
Thus $|\psi(x) (i)| \geq \varepsilon$ only if $i \in I_{n, \tau(n)^2}$ and $1/n\tau(n) \geq \varepsilon$.  There are only finitely many such $i$.

\begin{claim2}
For each $x \in \mathbf c_0$, we have $x \in \langle A \rangle \iff \psi(x) \in \langle K' \cap \mathbf c_0 \rangle$.
\end{claim2}

\proofclaim  To prove the claim, it will suffice to show that $\psi(x) \in \langle K' \cap \mathbf c_0 \rangle \iff \varphi (x) \in \langle K \rangle$, since we already have $x \in \langle A \rangle \iff \varphi(x) \in \langle K \rangle$.

Fix a group word $w$,
\begin{align*}
\varphi (x) \in w[K]
&\iff (\forall n) (\varphi (x) \upto n \in w[K] \upto n)\\
&\iff (\forall n) (\psi (x)  \upto I_{n , \tau(n)^2} \in (1/n\tau(n)^2)(w[K] \upto n))\\
&\iff \psi(x) \in w[K'].
\end{align*} 
The first and last ``$\iff$'' use the fact that $w[K]$ is closed (since $K$ is compact).  As $w$ was arbitrary, this completes the claim and proof.

{\bf $K_\sigma$ subgroups.}  Let $H = \bigcup_n K_n$ be a universal $K_\sigma$ subgroup of $\mathbb R^\omega$, as given by Theorem~\ref{T1}(1).  We may assume that 
\begin{equation}\label{E13}
(\bar 0 \in K_0) \mbox{ and } (\forall n) (-K_n = K_n \mbox{ and } K_n + K_n \subseteq K_{n+1}).
\end{equation}
Let $\{ I_{m,p} : m,p \in \omega\}$ be a family of intervals partitioning $\omega$ such that each $I_{m,p}$ has length $m$.  Define $K_n' \subseteq \mathbb R^\omega$ by 
\[
x \in K_n' \iff (\forall m,p) (x \upto I_{m,p} \in (1/mp) K_n \upto m)
\]
and let $H' = \bigcup K_n'$.  Again, Lemma~\ref{lemming} implies that each $K_n'$ is compact in $\mathbf c_0$.  Observe that \eqref{E13} holds for the $K_n'$ as well.  In particular, $H'$ is a subgroup of $\mathbb R^\omega$. We will show that $H' \cap \mathbf c_0$ is in fact a universal $K_\sigma$ subgroup of $\mathbf c_0$.

 Let $A = \bigcup_n A_n$ be an arbitrary $K_\sigma$ subgroup of $\mathbf c_0$.  Again, $A$ is still $K_\sigma$ in $\mathbb R^\omega$.  Hence there is a continuous homomorphism $\varphi : \mathbb R^\omega \rightarrow \mathbb R^\omega$ such that $\varphi^{-1} (H) = A$.  Let $\tau(m) \in \omega\setminus \{0\}$ be such that, for each $x \in [-1,1]^\omega$ and $i < m$, we have $|\varphi(x)(i)| \leq \tau(m)$.  Define $\psi : \mathbb R^\omega \rightarrow \mathbb R^\omega$ by 
\[
\psi(x) \upto I_{m,p} = \begin{cases}
(1/mp) \varphi(x) \upto m \quad &\mbox{if } p = \tau (m)^2,\\
0^m &\mbox{otherwise}.
\end{cases}
\] 
As in previous case, it follows that $\psi (\mathbf c_0) \subseteq \mathbf c_0$.  Finally, to see that $\psi^{-1} (H') = A$, it will suffice to show that 
\[
(\forall x \in \mathbf c_0) (\forall n) (\psi (x) \in K_n' \iff \varphi(x) \in K_n).
\]
To see this, observe that, for each $n$, 
\begin{align*}
\psi(x) \in K_n'
&\iff (\forall m) (\psi(x) \upto I_{m,\tau(m)^2} \in (1/m\tau(m)^2) K_n \upto m\\
&\iff (\forall m) (\varphi (x) \upto m \in K_n \upto m)\\
&\iff \varphi(x) \in K_n.
\end{align*}
\end{proof}

We now proceed to the main result of this section.  The following definition may be found at the beginning of Gowers-Maurey \cite{GOWERS-MAUREY.unconditional}.

\begin{definition}
Let $\mathfrak B$ be an infinite-dimensional Banach space (over $\mathbb R$).  An {\em unconditional basis} for $\mathfrak B$ is a set $\{ e_n \}_{n \in \omega} \subseteq \mathfrak B$ such that
\begin{enumerate}
\item each $e_n$ is a unit vector,
\item for each $x \in \mathfrak B$, there is a unique sequence $a_0 , a_1 , \ldots \in \mathbb R$ with $x = \sum_{n \in \omega} a_n e_n$ (convergence in norm) and
\item any permutation of $\{ e_n \}_{n \in \omega}$ still has the previous property.
\end{enumerate}
\end{definition}

The following fact (also mentioned in \cite{GOWERS-MAUREY.unconditional}) gives a useful property of unconditional bases.

\begin{proposition}[\cite{GOWERS-MAUREY.unconditional}, Theorem 1]\label{basis-bound}
If $\{ e_n \}_{n \in \omega}$ is an unconditional basis for a Banach space $\mathfrak B$, then there is a constant $C$ such that for each $x = \sum_{n \in \omega} a_n e_n \in \mathfrak B$ and $(\varepsilon_n)_{n \in \omega} \in \mathbb [-1,1]^\omega$, we have
\[
\Big\lVert \sum_{n \in \omega} \varepsilon_n a_n \Big\rVert \leq C \Big\lVert \sum_{n \in \omega} a_n e_n\Big\rVert.
\]
\end{proposition}

The following lemma is consequence of this fact.

\begin{lemma}
If $\mathfrak B$ is an infinite-dimensional Banach space with an unconditional basis, then there are continuous linear maps $T_1 : \mathfrak B \rightarrow \mathbf c_0$ and $T_2 : \mathbf c_0 \rightarrow \mathfrak B$.
\end{lemma}

\begin{proof}
Let $\{ e_n \}_{n \in \omega}$ be an unconditional basis for $\mathfrak B$, with $C$ as in the previous proposition.  

We first show the existence of the map $T_1 : \mathfrak B \rightarrow \mathbf c_0$.  Define $T_1 : \mathfrak B \rightarrow \mathbf c_0$ by $T_1 (\sum_n a_n e_n) = (a_n)_{n \in \omega}$.  Since the sum $\sum_n a_n e_n$ is convergent, the sequence of partial sums is Cauchy.  Hence the norm of the $n$th term converges to 0.  It follows that $T_1$ maps $\mathfrak B$ into $\mathbf c_0$.  We must now see that $T_1$ is continuous.  Since $T_1$ is linear, it will suffice to show that $T_1$ is continuous at the zero element of $\mathfrak B$.  Fix $x = \sum_n a_n e_n \in \mathfrak B$.  For each $n$, let $\varepsilon_n = 1$ and $\varepsilon_k = 0$, for $k \neq n$, and observe that 
\[
|a_n| = \lVert a_n e_n \rVert = \Big\lVert \sum_{n \in \omega} \varepsilon_n a_n \Big\rVert \leq C \lVert x \rVert.
\]
Thus $\lVert T_1(x) \rVert_{\rm sup} \leq C \lVert x \rVert$, showing that $T_1$ is continuous at $0 \in \mathfrak B$.

We now proceed to the second claim.  Define $T_2 : \mathbf c_0 \rightarrow \mathfrak B$ by $T_2 ((a_n)_{n \in \omega}) = \sum_n \frac{a_n}{2^n} e_n$.  Since $(a_n)_{n \in \omega}$ is a bounded sequence, this latter sum is always well-defined.  To see that $T_2$ is continuous, observe that, if $\lVert (a_n)_{n \in \omega} \rVert_{\rm sup} \leq 1$, then by Proposition~\ref{basis-bound} 
\[
\lVert T_2 ((a_n)_{n \in \omega}) \rVert = \Big \lVert \sum_n {\textstyle \frac{a_n}{2^n}} e_n \Big \rVert \leq C \Big \lVert \sum_n {\textstyle \frac{1}{2^n}} e_n \Big\rVert.
\]
Thus $T_2$ is a bounded linear map and hence continuous.
\end{proof}

Combining this lemma with Proposition~\ref{P1}, we obtain the following theorem.

\begin{theorem}\label{unconditional}
Let $\mathfrak B$ be an infinite-dimensional Banach space with an unconditional basis.  Then $\mathfrak B$ has universal compactly generated and $K_\sigma$ subgroups.
\end{theorem}

\begin{remark*}
To put this theorem in context, recall that (among many others) all $\ell^p$ spaces ($1 \leq p < \infty$) have unconditional bases.  (On the other hand, Per Enflo \cite{ENFLO.approximation.problem} and later Gowers-Maurey \cite{GOWERS-MAUREY.unconditional} showed that there exist separable Banach spaces with no unconditional bases.)
\end{remark*}

The following serves as an addendum to the last theorem.

\begin{theorem}\label{moreBanach}
The following Banach spaces (viewed as topological groups) have universal compactly generated and $K_\sigma$ subgroups:
\begin{enumerate}
\item $\ell^\infty$,
\item $C(X)$, if $X$ is infinite, Polish and compact, and
\item $C_0 (X)$, if $X$ is infinite, Polish and locally compact.
\end{enumerate}
\end{theorem}

\begin{remark*}
In general, the spaces listed in this theorem may not have unconditional bases ($\ell^\infty$ is not even separable) and so Theorem~\ref{unconditional} does not necessarily apply.
\end{remark*}

\begin{proof}[Proof of Theorem~\ref{moreBanach}]
In each case, we will apply Proposition~\ref{P1} and Theorem~\ref{c_0}.

1.  The injection $\mathbf c_0 \rightarrow \ell^\infty$ is via the inclusion map, while the injection $\ell^\infty \rightarrow \mathbf c_0$ is by means of the map $(a_n)_{n \in \omega} \mapsto ((1/n)a_n)_{n \in \omega}$.

2.  Let $\{ x_n \}_{n \in \omega}$ be a discrete sequence of distinct points in $X$.  For each $n$, let $f_n \in C(X)$ have sup-norm 1 and be such that $f_n(x_n) = 1$ and $f_n(x_k) = 0$, if $k \neq n$.  Such functions exist by the Tietze Extension Theorem.  Then $\mathbf c_0$ may be one-to-one homomorphically mapped into $C(X)$ via the continuous function $x \mapsto \sum_{n \in \omega} (x(n) / 2^n) f_n$.

Let $\{ y_n \}_{n \in \omega}$ be a countable dense subset of $X$.  Then $C(X)$ is injected into $\mathbf c_0$ via the map $f \mapsto ((1/n)f(y_n))_{n \in \omega}$.

3.  Use the same functions as in 2.
\end{proof}

\subsection{A negative example}  The following example gives our only instances of perfect Polish groups without universal subgroups in either of the classes we consider.  The key fact is that any nontrivial group homomorphism of $\mathbb R^n$ is in fact an automorphism.

\begin{example}
By Theorem~\ref{T1} there is a universal $K_\sigma$ subgroup of $\mathbb R^\omega$.  On the other hand, we shall see that there is no universal $K_\sigma$ subgroup of $\mathbb R^n$, for $n \in \omega$.  First, if $\varphi : \mathbb R^n \rightarrow \mathbb R^n$ is a continuous group homomorphism, then $\varphi$ is automatically a linear transformation.  To see this, observe that, since $\varphi$ is a group homomorphism, one can show that $\varphi (q\mathbf r) = q \varphi (\mathbf r)$, for any $q \in \mathbb Q$ and $\mathbf r \in \mathbb R^n$.  One then concludes that $\varphi (a\mathbf r) = a \varphi (\mathbf r)$, for any $a \in \mathbb R$, by the density of $\mathbb Q$ in $\mathbb R$ and the continuity of $\varphi$.

Towards a contradiction, suppose that $H_0 \subseteq \mathbb R^n$ is a universal $K_\sigma$ subgroup of $\mathbb R^n$.  Let $\tilde A, \tilde B \subsetneq \mathbb R$ be nontrivial $K_\sigma$ subgroups such that $\tilde A$ is countable and $\tilde B$ is uncountable.  Let 
\[  
A = \{ (x_1 , \ldots , x_n ) \in \mathbb R^n : x_1 \in \tilde A \ \& \ x_2 = x_3 = \ldots = x_n = 0\}
\] 
and
\[  
B = \{ (x_1 , \ldots , x_n ) \in \mathbb R^n : x_1 \in \tilde B \ \& \ x_2 = x_3 = \ldots = x_n = 0\}.
\] 

$A$ and $B$ are $K_\sigma$ subgroups of $\mathbb R^n$ that contain no linear (over $\mathbb R$) subspaces of $\mathbb R^n$ other than $\{ 0^n \}$.  Let $\varphi_A$, $\varphi_B$ be continuous endomorphisms of $\mathbb R^n$ reducing $A$, $B$ to $H_0$.  As $\varphi_A$ and $\varphi_B$ are actually linear transformations, $\ker (\varphi_A)$ and $\ker (\varphi_B)$ are linear subspaces of $\mathbb R^n$.  Since $\varphi_A$ and $\varphi_B$ are reductions between subgroups, we must have that $\ker (\varphi_A )\subseteq A$ and $\ker (\varphi_B) \subseteq B$, in particular, both kernels are trivial.  Hence $\varphi_A$ and $\varphi_B$ are actually automorphisms.  Thus $A$ and $B$ have the same cardinality, a contradiction.

By the same reasoning, there are no universal compactly generated or $K_\sigma$ subgroups of $\mathbb R^n$.
\end{example}



\section{An application to ideals}\label{S8}

Recall that an \emph{ideal} on $\omega$ is a set $\mathcal I \subseteq \mathcal P (\omega)$ that is closed under finite unions and closed downwards (i.e., if $x \subseteq y \in \mathcal I$, then $x \in \mathcal I$).  Also recall that $\mathcal P(\omega)$ becomes a Polish group when equipped with the addition operation
\[
x \btu y = (x \setminus y ) \cup (y \setminus x).
\]
In particular, every ideal is a subgroup of $\mathcal P (\omega)$, since $x \btu y \subseteq x \cup y$, for $x,y \subseteq \omega$.

By identifying each $x \subseteq \omega$ with its characteristic function, one can regard $(\mathcal P(\omega) , \btu)$ as $(\mathbb Z_2^\omega , +)$.  With this identification, the relation $x \subseteq y$ agrees with the pointwise $x \leq y$.  We use the latter when dealing with $\mathbb Z_2^\omega$ to avoid confusion with the ``$\,\subset\,$'' (extension) relation on $\mathbb Z_2^{<\omega}$.

In this section, we study the following weak form of Rudin-Keisler reduction.

\begin{definition}\label{weakRK}
For ideals $\mathcal I, \mathcal J$ on $\omega$, we write $\mathcal I \leqRK^+ \mathcal J$ if, and only if, there is a subset $A \subseteq \omega$ and a function $\beta : A \rightarrow \omega$ such that $x \in \mathcal I \iff \beta^{-1}(x) \in \mathcal J$, for each $x \subseteq \omega$.\footnote{We use the notation $\leqRK^+$ as a parallel with $\leqRB$ versus $\leqRB^+$.  See Kanovei \cite[pp.~41-42]{KAN.borel.equiv} for definitions.}
\end{definition}

Theorem~\ref{univeffideal} will use the method of Theorem~\ref{T1} to show that there is a $\leqRK^+$-complete $F_\sigma$ ideal.  In a personal communication, Michael Hru\u s\' ak has informed us that, though unpublished, this result is already known to him, albeit in a slightly different form.\footnote{See Hru\u s\' ak \cite[5.4]{HRUSAK.combinatorics.filters} for a similar result.}

The only difference between $\leqRK^+$ and the usual Rudin-Keisler order is that the reducing map in the case of $\leqRK^+$ need not be defined on all of $\omega$.  As with Rudin-Keisler reduction, if $\mathcal I \leqRK^+ \mathcal J$ and $\mathcal J$ is an ideal, then $\mathcal I$ is an ideal as well.  We call a map $\beta$ as in the definition above a \emph{weak RK-reduction.}  Observe that the map 
\[
x \mapsto \beta^{-1} (x)
\]
defines a continuous homomorphism of $\mathcal P (\omega)$ (equivalently, of $\mathbb Z_2^\omega$).  This implies that, for ideals $\mathcal I , \mathcal J$, if $\mathcal I \leqRK^+ \mathcal J$, then automatically $\mathcal I \leqg \mathcal J$.

Before proceeding, we verify that $\leqRK^+$ is indeed weaker than $\leqRK$.  Consider the following example.

\begin{example} For $x \subseteq \omega$, let 
\[
\mathrm{Fin}(x) = \{ y \in \mathcal P(\omega) : y \mbox{ is finite and } y \subseteq x \}.
\]
With this notation, the ideal $\mathrm{Fin}$ is $\mathrm{Fin}(\omega)$.  If $x$ is infinite, then any bijection $\beta : x \longleftrightarrow \omega$ witnesses $\mathrm{Fin} \leqRK^+ \mathrm{Fin}(x)$.  On the other hand, if $x \neq \omega$, then $\mathrm{Fin} \nleq_{\rm RK} \mathrm{Fin}(x)$. To see this, suppose otherwise and let $\beta : \omega \rightarrow \omega$ be such for each $y \subseteq \omega$, $y \in \mathrm{Fin} \iff \beta^{-1}(y) \in \mathrm{Fin}(x)$.  Let $a \in \omega \setminus x$ and let $b = \beta(a)$.  We have $\{ b \} \in \mathrm{Fin}$, but $\beta^{-1}(\{ b \}) \notin \mathrm{Fin}(x)$, since $a \in \beta^{-1}(\{ b \})$ and $a \notin x$.
\end{example}

We also remark on the fact that $\leqg$ is weaker than $\leqRK^+$.

\begin{example} Consider $H = \{ \emptyset , \{ 0,1 \} \}$ and the ideal $\mathrm{Fin}$.  Both are subgroups of $(\mathcal P(\omega) , \btu)$ and $H \leqg \mathrm{Fin}$, via the map $\varphi : \mathcal P(\omega) \rightarrow \mathcal P(\omega)$ defined by 
\[ 
\varphi (x) = \begin{cases}
\emptyset \quad &\mbox{if } 0,1 \in x \mbox{ or both } 0,1 \notin x,\\
\omega &\mbox{otherwise.}
\end{cases}
\]
It is easier to see that this is a group homomorphism by viewing $\mathcal P(\omega)$ as $\mathbb Z_2^\omega$.  With this identification, $\varphi$ is given by 
\[ 
\varphi (x)(n) = x(0) + x(1),
\]
for all $x \in \mathbb Z_2^\omega$ and $n \in \omega$.

On the other hand, we cannot have $H \leqRK^+ \mathrm{Fin}$, since this would imply that $H$ is an ideal.
\end{example}

We now proceed to the main result of this section.

\begin{theorem}\label{univeffideal}
There is a $\leqRK^+$-complete $F_\sigma$ ideal in $\mathbb Z_2^\omega$.
\end{theorem}

\begin{remark*}
Since every ideal on $\omega$ is a subgroup of the compact group $\mathbb Z_2^\omega$, Theorem~\ref{T4} implies that every $F_\sigma$ (i.e., $K_\sigma$) ideal is compactly generated.  Since the downward closure of a compact set is also compact, we conclude that every $F_\sigma$ ideal on $\omega$ is the set of finite unions of elements of a downward closed compact subset of $\mathcal P (\omega)$.
\end{remark*}

\begin{proof}[Proof of Theorem~\ref{univeffideal}]
For $k\in \omega$ and $s \in \omega^{<\omega}$, let $A^k_s$ be subsets of $\mathbb Z_2^k$ such that 
\begin{itemize}
\item Each $A^k_s$ is closed downward, i.e., $u \leq v \in A^k_s \implies u \in A^k_s$.

\item If $A \subseteq \mathbb Z_2^k$ is closed downward and $A \supseteq A^k_s$, then there exists $i$ such that $A = A^k_{s \concat i}$.
\end{itemize}
For each $k,j$, let $I^k_j$ be an interval in $\omega$ of length $k$, such that the $I^k_j$ partition $\omega$.  Define $A \subseteq \mathbb Z_2^\omega$ by 
\[ 
x \in A \iff (\exists n) (\forall k,s) (\len s \geq n \implies x \upto I^k_s \in A^k_s).
\]
Observe that $A$ is $F_\sigma$ and hence so is the ideal $\mathcal I_0$, generated by $A$.  Note that $A$ is already closed downward and thus $\mathcal I_0$ is the set of finite unions of elements of $A$.  We will show that $\mathcal I_0$ is $\leqRK^+$-complete among $F_\sigma$ ideals.

Let $\mathcal I = \bigcup_n F_n$ be an arbitrary $F_\sigma$ ideal.  We may assume that $F_0 \subseteq F_1 \subseteq \ldots$ and that each $F_n$ is closed downward.  (Since the downward closure of a closed set is also closed.)  For each $k$, choose $\alpha_k \in \omega^\omega$ such that for each $n$,
\[ 
F_n \upto k = A^k_{\alpha_k \upto n}.
\]
Let $S = \bigcup_{s \subset \alpha_k} I^k_s$.  We will define a weak RK-reduction $\beta : S \rightarrow \omega$ which will witness $\mathcal I \leqRK^+ \mathcal I_0$.  For each $I^k_s$, with $s \subset \alpha_k$, if $i$ is the $j$th element of $I^k_s$, we set $\beta (i) = j$.  We can re-write the map $x \mapsto \beta^{-1}(x)$ in a way that will be easier to work with.  Observe that 
\[ 
\beta^{-1}(x) \upto I^k_s =
\begin{cases}
x \upto k \quad &\mbox{if } s \subset \alpha_k,\\
0^k &\mbox{otherwise.}
\end{cases}
\]
The following two claims will complete the proof.

\begin{claim1}
If $x \in \mathcal I$, then $\beta^{-1}(x) \in \mathcal I_0$.
\end{claim1}

\proofclaim  Suppose that $x \in \mathcal I$, with $x \in F_{n_0}$.  This implies that, for each $k$ and $s \subset \alpha_k$, with $n=\len s \geq n_0$, we have
\begin{align*}
\beta^{-1} \upto I^k_s
&= x \upto k \\
&\in F_n \upto k \\
&= A^k_{\alpha_k \upto n}.
\end{align*}
If $s \not\subset \alpha_k$, then $\beta^{-1} (x) \upto I^k_s = 0^k \in A^k_s$, since $A^k_s$ is closed downwards.  Putting these two cases together, we see that
\[ 
(\forall k,s) (\len s \geq n_0 \implies \beta^{-1}(x) \upto I^k_s \in A^k_s).
\]
Hence $\beta^{-1} (x) \in A \subseteq \mathcal I_0$.  This proves our first claim.

\begin{claim2}
If $\beta^{-1}(x) \in \mathcal I_0$, then $x \in \mathcal I$.
\end{claim2}

\proofclaim  Suppose that $\beta^{-1}(x) \in \mathcal I_0$ and $y_1, \ldots , y_m \in A$ are such that $\beta^{-1}(x) = y_1 \cup \ldots \cup y_m$.  We will find $x_1 , \ldots , x_m \in \mathcal I$ such that $x = x_1 \cup \ldots \cup x_m$.  Let $n$ be such that for each $i \leq m$, 
\[ 
(\forall k,s) (\len s \geq n \implies y_i \upto I^k_s \in A^k_s).
\]
Let $v^k_i = y_i \upto I^k_{\alpha_k \upto n}$.  For each $k$ and all $i \leq m$, $v^k_i \in A^k_{\alpha_k \upto n} = F_n \upto k$.  Hence there exists $x^k_i \in F_n$ such that $v^k_i = x^k_i \upto k$.  By repeated use of the compactness of $\mathbb Z_2^\omega$, we choose a subsequence $k_0 < k_1 < \ldots$ and $x_i \in F_n$ such that, for each $i \leq m$
\[ 
\lim_{p \rightarrow \infty} x^{k_p}_i = x_i.
\]
To check that $x = x_1 \cup \ldots \cup x_m$, observe that, for each fixed $\ell$ and $p$ with $k_p \geq \ell$, we have
\begin{align*}
x \upto \ell 
&= (\beta^{-1}(x) \upto I^{k_p}_{\alpha_{k_p} \upto n}) \upto \ell \\
&= (v^{k_p}_1 \cup \ldots \cup v^{k_p}_m) \upto \ell \\
&= (x^{k_p}_1 \cup \ldots \cup x^{k_p}_m) \upto \ell .
\end{align*}
Taking the limit as $p \rightarrow \infty$, we see that 
\[ 
x \upto \ell = (x_1 \cup \ldots \cup x_m) \upto \ell. 
\]
Since $\ell$ was arbitrary, we must have $x = x_1 \cup \ldots \cup x_m$.  This shows that $x \in \mathcal I$ and completes the proof.
\end{proof}



\bibliographystyle{plain}
\bibliography{references}
\end{document}